\documentclass[11pt]{amsart}
\usepackage{mathrsfs}
\usepackage{amsfonts}
\usepackage{amsmath}
\usepackage{amssymb}
\usepackage{amsthm}
\usepackage{enumerate}
\usepackage{color}
\usepackage{geometry}
\usepackage{hyperref}
\usepackage{float}
\usepackage{graphicx}
\usepackage{multirow}
\usepackage[numbers,sort&compress]{natbib}
\usepackage{color}

\allowdisplaybreaks
\numberwithin{equation}{section}

\theoremstyle{plain}
\newtheorem{prop}{Proposition}[section]
\newtheorem{coro}[prop]{Corollary}

\newtheorem{lemm}[prop]{Lemma}

\newtheorem{theorem}[prop]{Theorem}
\newtheorem*{thrmA}{Theorem A}
\newtheorem*{thrmB}{Theorem B}
\newtheorem*{thrmC}{Theorem C}

\theoremstyle{definition}
\newtheorem{defi}[prop]{Definition}

\newtheorem{exam}[prop]{Example}

\renewcommand\aa{a}

\newcommand\bb{b}

\newcommand\cc{c}
\newcommand{\chart}{\mathsf{char}}

\newcommand\D[1]{|#1|}
\newcommand\da{associative differential commutative algebra}
\newcommand\adsa{associative differential supercommutative algebra}
\newcommand\dd{d}
\newcommand\DDX{\oD_\mathsf{s}^{0}\NB\XX}
\newcommand\DX{\oD_\mathsf{s}\NB\XX}
\newcommand\fdsa{\kk_{\mathsf{s}}\{\XX\}}
\newcommand\ff{f}
\newcommand\ffb{\lm{\ff}} 
\newcommand\fns[1]{\mathsf{GDN}_{\mathsf{s}}(#1)}

\renewcommand\geq{\geqslant}

\newcommand\HS[1]{\leavevmode\null\hspace{#1mm}}
\newcommand\ie{{\it i.e.}}
\newcommand\ii{i}

\newcommand\id{\mathsf{Id}}
\newcommand\Id[1]{\id(#1)}
\newcommand\IdD[1]{\id_{D}[#1]}
\newcommand\IDD[1]{\id^{0}_{D}[#1]}
\newcounter{ITEM}
\newcommand\ITEM[1]{\setcounter{ITEM}{#1}\leavevmode\hbox{\rm(\roman{ITEM})}}
\newcommand\jj{j}
\newcommand\kk{k}
\newcommand\kds[1]{\kk\oD_{\mathsf{s}}\NB{\XX |#1}}
\newcommand\kdds[1]{\kk\oD_{\mathsf{s}}^{0}\NB{\XX |#1}}
\renewcommand\leq{\leqslant}

\newcommand\lcoe[1]{\mathsf{lc}(#1)}

\renewcommand\ll{l}
\newcommand\lnormed[2]{[#1\wdots#2]{_{_\mathsf{L}}}}

\newcommand\lm[1]{\overline{#1}}
\newcommand\LSUB{_{\mathsf{L}}}

\newcommand\mA{\mathcal{A}}

\newcommand\mm{m}
\newcommand\mB{\mathcal{B}}
\newcommand\mV{\mathcal{V}}

\newcommand\mr[1]{\mathsf{r}(#1)}

\newcommand\nn{n}

\newcommand\ns[2]{\mathsf{GDN}_{\mathsf{s}}(\XX|\SS)}
\newcommand{\nsup}{GDN superalgebra}
\newcommand{\nsups}{GDN superalgebras}
\newcommand{\nsupt}{GDN supertableau}

\newcommand{\nsuptx}{GDN supertableaux}
\newcommand\NB[1]{[#1]} 
\newcommand\oD{D}

\newcommand\ov[1]{\overline{#1}}
\newcommand\pdots{\mathrel{\HS{0.2}{\cdot}{\cdot}{\cdot}\HS{0.2}}}

\newcommand\pp{p}
\newcommand\qq{q}
\newcommand\quldots{ ...}
\newcommand\rr{r}

 \newcommand\rnormed[2]{[#1\wdots#2]{_{_\mathsf{R}}}}

\renewcommand\ss{s}
\renewcommand\SS{S}

\newcommand\tab{\mathsf{Tab}_{\mathsf{s}}(X)}
\newcommand\tl[1]{\widetilde{#1}}
\renewcommand\tt{t}
\newcommand\uu{u}

\newcommand\vv{v}
\newcommand\wdots{, ...\HS{0.2}, }

\newcommand\wt[1]{\mathsf{wt}#1}
\newcommand\xx{x}
\newcommand\XX{X}

\newcommand\XXps{\XX^{(\ast)}}

\newcommand\YY{Y}
\newcommand\YYplus{\YY^+}
\newcommand\yy{y}
\newcommand\zz{z}
\newcommand\ZZ{Z}

\title{On free Gelfand--Dorfman--Novikov superalgebras  and  a PBW type theorem$^{\ddag}$}

\author{Zerui Zhang$^*$}
\address{Z.Z., School of Mathematical Sciences, South China Normal University, Guangzhou 510631, P. R. China}
\email{\small 295841340@qq.com}

\author{Yuqun Chen$^{\sharp}$}
\address{Y.C., School of Mathematical Sciences, South China Normal University, Guangzhou 510631, P. R. China}
\email{yqchen@scnu.edu.cn}

\author{Leonid A. Bokut$^{\dagger}$}
\address{L.A.B., School of Mathematical Sciences, South China Normal University Guangzhou 510631, P. R. China; Sobolev Institute of mathematics, Novosibirsk, 630090, Russia; Novosibirsk State University, Novosibirsk 630090, Russia}
\email{bokut@math.nsc.ru}

\thanks{${}^{\ddag}$Supported by the NNSF of China (11571121), the NSF of Guangdong Province (2017A030313002) and the Science and Technology Program of Guangzhou (201707010137)}

\thanks{${}^*$ Supported by the Innovation Project of Graduate School of South China Normal University}

\thanks{${}^{\sharp}$ Corresponding author}

\thanks{${}^{\dagger}$ Supported by Russian Science Foundation (project 14-21-00065)}



\begin{document}

\begin{abstract}
We construct linear bases of free GDN superalgebras.
 As applications,
we prove a Poincar\'{e}--Birkhoff--Witt type theorem, that is, every GDN superalgebra can be  embedded into
its universal enveloping associative differential supercommuative algebra.  An Engel theorem is  given.
\end{abstract}

\maketitle

\section{Introduction}\label{Intro}
We recall that a superalgebra over a field~$\kk$ is a vector space~$\mA$ with a direct sum decomposition
$\mA=\mA_0\oplus \mA_1$
together with a bilinear multiplication $\circ$: $\mA \times \mA\mapsto\mA$ such that~$\mA_i\circ \mA_j\subseteq \mA_{i+j}  $,
where the subscripts are elements of~$\mathbb{Z}_2$. The \emph{parity}~$\D\xx$ of every element~$\xx$ in~$\mA_0$ is~0, and the \emph{parity}~$\D\xx$ of every nonzero element~$\xx$ in~$\mA_1$ is 1.
 If a superalgebra~$\mA$ satisfies the following two identities
 $$
(x\circ(y\circ z))-((x\circ y)\circ z)
=(-1)^{|x||y|}((y\circ(x\circ z))-((y\circ x)\circ z))\ \ \ (\mbox{left supersymmetry}),
$$
and
$$
((x\circ y)\circ z)=(-1)^{|y||z|}((x\circ z)\circ y) \ \ \ (\mbox{right supercommutativity})
$$
for all elements~$\xx,\yy,\zz$ in~$\mathcal{A}_0\cup \mathcal{A}_1$,
then~$\mA$ is called a (left) \emph{Novikov superalgebra}~\cite{Xiaoping Xu super}. (There is a ``right" version of using right supersymmetry and left supercommutativity.) Moreover, if a (left) Novikov superalgebra~$\mA$ equals to its even part, \ie, $\mA=\mA_0$, then~$\mA$ is just an ordinary (left) \emph{Novikov algebra}~\cite{Gelfand,Novikov1,Novikov2}. Since Novikov algebras were invented by Balinskii and Novikov~\cite{Novikov1}, and independently by Gelfand and Dorfman~\cite{Gelfand}, we
also call Novikov algebra as Gelfand--Dorfman--Novikov algebra (GDN algebra) and call
Novikov superalgebra as Gelfand--Dorfman--Novikov superalgebra (GDN superalgebra).

 A rich structure and combinatorial theory of GDN algebras have been done up to now. Zelmanov solved Novikov's problem on classification of simple GDN algebras over an algebraically closed field: There are no such algebras besides trivial~\cite{Zelmanov}. Osborn and Zelmanov classified simple GDN algebras $A$ over an algebraically closed field of
characteristic 0 with a maximal subalgebra $H$ such that~$A /H$ has a finite dimensional irreducible $H$-submodule~\cite{Osborn 5}. Xu gave a complete classification of finite dimensional simple GDN algebras and their irreducible modules over an algebraically closed field with prime characteristic~\cite{Xiaoping Xu 2}, and he introduced some quadratic GDN superalgebras connecting with Gelfand--Dorfman ($\Omega$-bi) algebras~\cite{Xiaoping Xu conformal} (Gelfand--Dorfman ($\Omega$-bi) algebras were invented in~\cite{Gelfand}).
See also, for example, Bai and Meng~\cite{Bai 1, Bai 3}, Burde and Dekimpe~\cite{Burde 3}, Chen,  Niu and Meng~\cite{Chen ly}, Kang and Chen~\cite{super low dimension}, Zhu and Chen~\cite{A001},  Bokut,  Chen and Zhang~\cite{GDN,GDNP}.

 Dzhumadildaev and  L\"{o}fwall proved that
 the set of all the Novikov tableaux (we call them GDN tableaux because of the above reason) over a well-ordered set~$\XX$ forms a linear basis of a free GDN algebra generated by~$\XX$ by using trees and by appealing to the connection with free commutative associative differential algebra~\cite{trees}
(the idea of this connection was given by S.I. Gelfand, see~\cite{Gelfand}).
 And we wonder what would a basis of a free GDN superalgebra be like. The method of using trees developed in~\cite{trees} can not be directly applied for GDN superalgebras, but the idea of tracing a root of a tree can be modified to define the root number of a term. Moreover, the definition of GDN tableau can be easily extended to a definition of GDN supertableau, see Definition~\ref{def-sup-tab}. One of the results we prove below is as follow:
\begin{thrmA}\label{THA}
The set of all the GDN supertableaux over a well-ordered set~$\XX=\XX_0\cup \XX_1$ forms a linear basis of the
free GDN superalgebra $\fns\XX$ generated by~$\XX$, where every element of the set~$\XX_0$ is of parity~0 and every element of the set~$\XX_1$ is of parity~1.
\end{thrmA}
We also prove a Poincar\'{e}--Birkhoff--Witt (PBW) type theorem for GDN superalgebras:
 \begin{thrmB}
Every GDN superalgebra can be  embedded into
its universal enveloping associative differential supercommutative algebra.
 \end{thrmB}
As a corollary, we show that every GDN superalgebra generated by a finite set of elements of parity~1 is nilpotent. Several results concerning the nilpotency of certain GDN algebras have been found up to now.
Zelmanov  proved that, if~$\mA$ is a left-nilpotent finite dimensional (right) GDN
algebra  over a field of characteristic zero, then~$\mA^{2}$ is nilpotent~\cite{Zelmanov}.
Filippov proved that a right-nil algebra of bounded index over a field of characteristic
zero is right nilpotent provided that it is right symmetric  and is nilpotent provided that it is a right GDN algebra~\cite{Filippov 2}.
  Dzhumadildaev and  Tulenbaev  proved that if a
(right) GDN algebra~$\mA$   over a field of characteristic~$\pp$ is left-nil of bounded index~$\nn$ and~$\pp=0$ or~$\pp>\nn$, then~$\mA^2$ is nilpotent~\cite{engel novikov}. Again, to some extent, this kind of result can be extended to the case of GDN superalgebras, and we prove the following Engel theorem:
\begin{thrmC}\label{theoremc}
 Let~$\mA=\mA_0\oplus \mA_1$ be a $($left$)$ GDN superalgebra over a field of characteristic~0 generated by~$\XX=\XX_0\cup \XX_1$,
 where every element of the set~$\XX_0$ is of parity~0 and every element of the set~$\XX_1$ is of parity~1. If for some integer~$\nn>0$, the even part~$\mA_0$ is right-nil of bounded index~$\nn$ and $\XX_1$ is a finite set, then~$\mA^2$ is nilpotent.
\end{thrmC}

The paper is organized as follows.
In section~\ref{sec-gdnsupertableaux}, we construct a linear generating set~$\tab$ for a free GDN superalgebra generated by a well-ordered set~$\XX$ over a field of characteristic~$\neq2$. (For  the  case of characteristic~$2$, a GDN superalgebra is the same as a GDN algebra, so a linear basis is already known~\cite{trees}). In section~\ref{basis-and-pbw-sec}, we show the linear independence of~$\tab$, and we also prove a PBW type theorem for~\nsups. In section~\ref{nil-sec}, we prove Theorem~C, an Engel type theorem for GDN superalgebras.
\section{A linear generating set of~$\fns\XX$} \label{sec-gdnsupertableaux}
Our aim in this section is to construct a specific linear generating set of the free GDN superalgebra~$\fns\XX$ generated by a well-ordered set~$\XX$. (We shall show that the set we constructed is indeed linearly independent in the next section.) The idea of our construction is reminiscent of what was done for GDN algebras in~\cite{trees}. However, the original method of using trees is not extended directly. So we develop a new notion of root number of a term. In the whole paper, we assume that~$\XX=\XX_0\cup\XX_1$ is a fixed well-ordered set, where every element of the set~$\XX_0$ is of parity~0 and every element of the set~$\XX_1$ is of parity~1.
We also assume that the characteristic~$\chart(\kk)$ of the field~$\kk$ is not 2.

\subsection{The root number of a term}
In this subsection, we first recall the definition of terms. Then we define the root map from the set of all terms over a set~$\XX$ to the set of nonnegative integers. Finally, we develop several handy properties of the root map. They will be useful in the sequel when we try to develop a method of writing an arbitrary term into a linear combination of some specified terms (hereafter called GDN supertableaux).

We recall that \emph{terms} over~$\XX$ are defined by the following induction:

\ITEM1 Every element~$\aa$ of~$\XX$ is a term over~$\XX$;

\ITEM2 If~$\mu$ and~$\nu$ are terms over~$\XX$, then~$(\mu \circ \nu)$ is a term over~$\XX$. Denote by~$\XXps$ the set of all terms over~$\XX$.

For every term~$\mu$ in~$\XXps$, the \emph{length}~$\ell(\mu)$ of~$\mu$ is defined to be~1 if~$\mu$ lies
in~$\XX$, and~$\ell(\mu)$ is defined to be~$\ell(\mu_1)+\ell(\mu_2)$ if~$\mu=(\mu_1\circ\mu_2)$ for some terms~$\mu_1$ and~$\mu_2$ in~$\XXps$.
Similar to the definition of length, for every term~$\mu$ in~$\XX^{(\ast)}$, the parity~$\D\mu$ of~$\mu$ satisfies the following claims: \ITEM1 $\D\mu=0$ if~$\mu$ lies in~$\XX_0$, and $\D\mu=1$ if~$\mu$ lies in~$\XX_1$. \ITEM2 $\D\mu=\D{\mu_1}+\D{\mu_2}$ modulo~2 if~$\mu=(\mu_1\circ \mu_2)$.

\begin{defi}
 We define a \emph{root map}~$\mathsf{\rr}$ from the set~$\XXps$ to the set~$\mathbb{\ZZ}_{\geq 0}$ of nonnegative integers defined inductively as follows:

 \ITEM1  $\mr\aa=0$ for every element~$\aa$ in $\XX$;

\ITEM2  $\mr{(\mu \circ \nu )}=\mr{\mu}+1$ if~$\nu$ lies in~$\XX$, and
        $\mr{ (\mu \circ \nu )}=\mr\mu+\mr\nu$ if~$\nu$ does not lie in~$\XX$.
\end{defi}
For every term~$\mu$ in~$\XXps$, we call~$\mr\mu$ the \emph{root number} of~$\mu$ to indicate that our idea is based on~\cite{trees}, in which the authors appealed to the tool of trees (and roots of trees).
For all terms~$\mu_1, \quldots , \mu_\nn $ in~$\XXps$, to make the notations shorter, define
\begin{align*}
\lnormed{\mu_1}{\mu_\nn}&=((\quldots ((\mu_1\circ \mu_2) \circ\mu_3) \circ\pdots)\circ \mu_\nn) \ \ \mbox{(left-normed bracketing),}&\\
\rnormed{\mu_1}{\mu_\nn}&= (\mu_1\circ(\pdots\circ (\mu_{\nn-2} \circ(\mu_{\nn-1} \circ\mu_{\nn}))\quldots )) \ \ \mbox{(right-normed bracketing).}&
\end{align*}
Moreover, if~$\mu_1\wdots\mu_n$ are elements of~$\XX$, then we call~$\rnormed{\mu_1}{\mu_\nn}$ a~\emph{simple term} over~$\XX$ of \emph{length}~$\nn$.

Below we offer an instance of counting the root number of a term in~$\XXps$.
\begin{exam}\label{example n-1}
 For every positive integer~$\nn$, for all elements~$\aa_\ii$ $(1\leq \ii\leq \nn)$ in~$\XX$,  we have

 \ITEM1  $\mr{\rnormed{\aa_1}{\aa_\nn}}=1$;

 \ITEM2  $\mr{\lnormed{\aa_1}{\aa_\nn}}=\nn-1$.
\end{exam}
In general, the root number of a term~$\mu$ is not uniquely decided by the length~$\ell(\mu)$. The following lemma shows that the root number~$\mr\mu$ is bounded above by~${\ell(\mu)-1}$.

\begin{lemm}{\label{prop n-1}}
For every term~$\mu$ in~$\XXps$, we have~$\mr\mu\leq \ell(\mu)-1$, with equality only if~$\mu=\lnormed{\aa_1}{\aa_{\ell(\mu)}}$ for some elements~$\aa_1\wdots\aa_{\ell(\mu)}$ in~$\XX$.
\end{lemm}
\begin{proof}
Use induction on~$\ell(\mu)$. For~$\ell(\mu)=1$, we have~$\mr\mu=0=\ell(\mu)-1$. For~$\ell(\mu)>1$, we have~$\mu=(\mu_1\circ\mu_2)$ for some terms~$\mu_1$ and~$\mu_2$ in~$\XXps$. If~$\ell(\mu_2)>1$, then by induction hypothesis, we have~$$\mr\mu=\mr{\mu_1}+\mr{\mu_2}\leq \ell(\mu_1)-1+\ell(\mu_2)-1<\ell(\mu)-1.$$
On the other hand, if~$\ell(\mu_2)=1$, then by induction hypothesis, we have~$$\mr\mu=\mr{\mu_1}+1\leq \ell(\mu_1)-1+1=\ell(\mu_1)=\ell(\mu)-1,$$
with the equality only if~$\mr{\mu_1}=\ell(\mu_1)-1$. The claim follows by induction hypothesis.
\end{proof}
The following lemma offers a formula for counting the root number of a term.
\begin{lemm}\label{root-fml}
  For every term~$\mu=(\mu_1,\mu_2)$ in~$\XXps$, we
   have~$$\mr\mu=\mr{\mu_1}+\max(1,\mr{\mu_2})\geq 1,$$ with~$\mr\mu=1$ only if~$\mu=\rnormed{\aa_1}{\aa_{\ell(\mu)}}$ for some elements~$\aa_1\wdots\aa_{\ell(\mu)}$ in~$\XX$.
\end{lemm}
\begin{proof}
 Use induction on~$\ell(\mu)$. For~$\ell(\mu)=2$, the terms~$\mu_1$ and~$\mu_2$ lie in~$\XX$, so the claim follows. For~$\ell(\mu)>2$,  if~$\mu_2$ lies in~$\XX$, then~$\max(1,\mr{\mu_2})=1$; if~$\ell(\mu_2)\geq 2$, then by induction hypothesis, we have~$\mr{\mu_2}\geq 1$ and so~$\max(1,\mr{\mu_2})=\mr{\mu_2}$.
 Therefore, we obtain~$$\mr\mu=\mr{\mu_1}+\max(1,\mr{\mu_2})\geq 1.$$ If the equality~$\mr\mu=1$ holds, then the induction hypothesis forces ~$\ell(\mu_1)=1$.  The claim follows.
\end{proof}

The following lemma shows that the root map is compatible with the right supercommutativity, and to some extent, the root map is also compatible with the product~$\circ$.
\begin{lemm}\label{root-pro}
For all terms~$\mu_1,\mu_2$ and~$\mu_3$ in~$\XXps$,
we have

 \ITEM1 $\mr{[\mu_1,\mu_2,\mu_3]_{_\mathsf{L}}}
 =\mr{[\mu_1,\mu_3,\mu_2]_{_\mathsf{L}}}$;

 \ITEM2 If~$\mr{\mu_1}>\mr{\mu_2}$, then $\mr{(\mu_1\circ\mu_3)}>\mr{(\mu_2\circ\mu_3)}$;

 \ITEM3 If~$\mr{\mu_1}>\mr{\mu_2}$ and~$\mr{\mu_1}>1$, then~$\mr{(\mu_3\circ\mu_1)}>\mr{(\mu_3\circ\mu_2)}$;

 \ITEM4 If~$\mr{\mu_1}>\mr{\mu_2}$ and~$\mr{\mu_1}=1$, then~$\mr{(\mu_3\circ\mu_1)}=\mr{(\mu_3\circ\mu_2)}$.
\end{lemm}
\begin{proof} The lemma follows immediately from Lemma~\ref{root-fml}.
\end{proof}

\subsection{GDN supertableaux}

Now we are ready to define the notion of a GDN supertableau, which is directly reminiscent of the notation of a GDN tableau. Our aim in this subsection is to show that, if~$\chart(\kk)\neq 2$, then the set of all GDN supertableaux over~$\XX$ forms a linear generating set of the free GDN superalgebra~$\fns\XX$.
 \begin{defi}\label{def-sup-tab} We call a term~$\mu$  a \textit{Gelfand--Dorfman--Novikov supertableau} (GDN supertableau) over~a well-ordered set~$\XX=\XX_0\cup\XX_1$ if, for some letter~$\aa$ in~$\XX$, for some nonnegative integer~$\nn$, and for some simple terms~$\mu_{i}=\rnormed{\aa_{\ii,\rr_\ii}}{\aa_{\ii,1}}$ ($1\leq \ii\leq \nn$) over~$\XX$ of length~$\rr_i\geq 1$, we have
\begin{eqnarray}\label{form-sup-tab}
\mu=\lnormed{\aa,\mu_1}{\mu_{n}}
\end{eqnarray}
   such that
the following conditions hold:

\ITEM1 The integers~$\rr_1\wdots\rr_n$ satisfy that~$r_1\geq  \cdots \geq r_n$;

\ITEM2  If $r_i=r_{i+1}$, then $ a_{i,1}\geq a_{i+1,1}$ holds;

\ITEM3  The inequality~$a\geq a_{1,r_{1}}\geq \cdots \geq a_{1,2}\geq a_{2,r_{2}}\geq \cdots \geq a_{2,2}\geq \cdots \geq a_{n,r_{n}}\geq \cdots \geq a_{n,2}$ holds;

\ITEM4   If for some integers~$\ii,\tt,\jj$ and~$\ll$ satisfying~$\ii\leq n$, $\tt\leq n$, $2\leq \jj\leq \rr_j$ and~$2\leq\ll\leq \rr_t$, the letters~$a_{i,  j}$ and~$a_{t,  l}$ lie in~$\XX_1$, then the inequality~$a_{\ii,\jj}\neq a_{t,  l}$ holds;

\ITEM5   If for some integer~$\ii\leq \nn-1$, the elements~$\aa_{\ii,1}$ and~$\aa_{\ii+1,1}$ lie in~$\XX_1$, and~$\rr_{\ii}=\rr_{\ii+1}$, then the inequality~$\aa_{\ii,1}\neq\aa_{\ii+1,1}$ holds.
\end{defi}

Every term of the form~\eqref{form-sup-tab} satisfying Points~\ITEM1-\ITEM3 is called a \emph{Novikov tableau}~\cite{trees} over~$\XX$, and we call it a GDN tableau because of the reason explained in the introduction. Denote by~$\tab$ the set of all the GDN supertableaux over~$\XX$.
  It is quite easy to show that every term in~$\XXps\subseteq \fns\XX$ of the form~\eqref{form-sup-tab} can be written as a linear combination of terms satisfying Points~\ITEM1 and~\ITEM2 by the right supercommutativity, but
what remains becomes complicated and we will need the notion of root number of a term.

  The strategy for rewriting is to apply the right supercommutativity and the left supersymmetry. Unfortunately, whenever we apply the left supersymmetry to a term, we shall get three other terms in return, and thus this process becomes complicated. So a simplified notation is needed. Because of this reason, we introduce the following notation.
\begin{defi}
  For all terms~$\mu$ and~$\nu$ in~$\XXps$ such that~$\mr\mu=\mr\nu$ and~$\ell(\mu)=\ell(\nu)$, for all nonzero elements~$\alpha$ and~$\beta$ in the field~$\kk$,  the polynomials~$\alpha\mu$ and~$\beta\nu$ in~$\fns\XX$ are said to be equivalent, denoted by~$$\alpha\mu \sim \beta\nu,$$ if~$\alpha\mu-\beta\nu=\sum_{\ii}\alpha_i\mu_i$ in~$\fns\XX$  for some elements~$\alpha_i$ in~$\kk$ and terms~$\mu_i$ in~$\XXps$ such that~$\mr\mu<\mr{\mu_i}$ and~${\ell(\mu)=\ell(\mu_i)}$ for every~$\ii$.
\end{defi}
It is clear that if~$\alpha\mu \sim \beta\nu$ and~$\beta\nu\sim\alpha'\mu'$, then we get~$\alpha\mu \sim\alpha'\mu'$.

Recall that for every element~$\aa$ of~$\XX$, the parity~$\D{\aa}$ of~$\aa$ is~$\ii$ if~$\aa$ lies in $\XX_i$ with $i=0,1$. Moreover, for all elements~$\aa_1\wdots\aa_n$ ($\nn\geq 1$) of~$\XX$, we define the parity~$\D{\aa_1...\aa_n}$ of the string~$\aa_1...\aa_n$ to be~$\D{\aa_1}+\pdots+\D{\aa_n}$ modulo 2, extended with~$\D\varepsilon=0$ for the empty string~$\varepsilon$.

 The following lemma shows that, for every simple term~$\rnormed{\aa_r}{\aa_1}$ with~$\rr\geq 3$, we can rearrange~$\aa_r\wdots\aa_2$ at the expense of adding a linear combination of terms of length~$\rr$ and with root numbers~$>1$. We shall see in Lemma~\ref{generating set} that the added terms do not increase the difficulty of rewriting an arbitrary term into a linear combination of GDN supertableaux.
\begin{lemm}\label{interchange leaves} For all elements~$\aa_1\wdots\aa_r$ in~$\XX$,
for every simple term~$\mu=\rnormed{\aa_\rr}{\aa_1}$, the following claims hold:

\ITEM1 For every integer~$\jj$ such that~$2\leq \jj<\rr$, we
have~$$\mu\sim(-1)^{\D{\aa_\jj}\D{\aa_{\jj+1}...\aa_\rr }}\rnormed{\aa_\jj,\aa_r}{\aa_{\jj+1},\aa_{\jj-1}\wdots\aa_1};$$

\ITEM2 For all integers~$\ii$ and~$\jj$ such that~$2\leq \jj<\ii \leq \rr$,
we have~$$\mu\sim  (-1)^{\D{\aa_\ii}\D{\aa_\jj\dots \aa_{\ii-1} }
 +\D{\aa_\jj}\D{\aa_{\jj+1}\dots \aa_{\ii-1} }} \rnormed{\aa_r\wdots\aa_{\ii+1},\aa_\jj,\aa_{\ii-1}}{\aa_{\jj+1},
 \aa_{\ii},\aa_{\jj-1}\wdots\aa_1}.
 $$
\end{lemm}
\begin{proof} We shall just prove Point~\ITEM1, because Point~\ITEM2 can be proved in a similar way.
 Assume~$\nu=\rnormed{\aa_{\jj-1}}{\aa_1}$. Then by the left supersymmetry, we have
\begin{multline*}
 \mu=(-1)^{\D{\aa_\jj}\D{\aa_{\jj+1}}}\rnormed{\aa_\rr}{\aa_{\jj+2},\aa_\jj,\aa_{\jj+1},\nu}
      + \rnormed{\aa_\rr}{\aa_{\jj+2},(\aa_{\jj+1}\circ \aa_\jj),\nu}\\
      -(-1)^{\D{\aa_\jj}\D{\aa_{\jj+1}}}\rnormed{\aa_\rr}{\aa_{\jj+2},(\aa_{\jj}\circ \aa_{\jj+1}),\nu}.
\end{multline*}
  Since \begin{multline*}
  \mr{\rnormed{\aa_\rr}{\aa_{\jj+2},(\aa_{\jj+1}\circ \aa_\jj),\nu}}=\mr{\rnormed{\aa_\rr}{\aa_{\jj+2},(\aa_{\jj}\circ \aa_{\jj+1}),\nu}}=2\\
  >1=\mr\mu=\mr{\rnormed{\aa_\rr}{\aa_{\jj+2},\aa_\jj,\aa_{\jj+1},\nu}},
  \end{multline*}
we have
$$
\mu\sim(-1)^{\D{\aa_\jj}\D{\aa_{\jj+1}}}\rnormed{\aa_\rr}{\aa_{\jj+2},\aa_\jj,\aa_{\jj+1},\nu}.
$$
By induction on~$\rr-\jj$, we obtain
$$
\mu\sim(-1)^{\D{\aa_\jj}\D{\aa_{\jj+1}}}
\rnormed{\aa_\rr}{\aa_{\jj+2},\aa_\jj,\aa_{\jj+1},\nu}
\sim(-1)^{\D{\aa_\jj}\D{\aa_{\jj+1}...\aa_{\rr}}}
\rnormed{\aa_\jj,\aa_r}{\aa_{\jj+1},\aa_{\jj-1}\wdots\aa_1}.
$$
The proof is completed.
\end{proof}

For every simple term~$\mu=\rnormed{\aa_\rr}{\aa_1}$ in~$\XXps$, for every integer~$\ii$ such that~$2\leq i\leq \rr$, we define
 $$
 \mu_{\hat{\aa_\ii}}= \rnormed{\aa_r\wdots\aa_{i+1},\aa_{i-1}}{\aa_1}
 $$
 and
 $$
 \mu_{\aa_\ii\mapsto \bb_\jj}= \rnormed{\aa_r\wdots\aa_{i+1},\bb_\jj,\aa_{i-1}}{\aa_1}.
 $$

The following lemma is crucial to the construction of a linear basis of the free GDN superalgebra~$\fns\XX$. It shows that, for the product of two simple terms, we can ``interchange" certain letters of the two simple terms in the sense of adding a linear combination of some nonessential terms. We shall see that, as a result of the following lemma, the set of all the GDN tableaux over~$\XX$ is not linearly independent in~$\fns\XX$ provided that~$\XX_1$ is nonempty and the characteristic of the field is not~2.

\begin{lemm}\label{interchange two row leaves}
For all elements~$\aa_1\wdots\aa_{\rr+1},\bb_1\wdots\bb_\mm$ $(\rr\geq 2,\mm\geq 2)$ in~$\XX$, for all integers~$\ii,\jj$ such that $2\leq \ii\leq \rr+1$ and~$2\leq \jj\leq \mm$, for all simple terms~$\mu=\rnormed{\aa_{\rr+1}}{\aa_1}$ and~${\nu=\rnormed{\bb_{\mm}}{\bb_1}}$,  we can interchange $\aa_i$ and~$\bb_\jj$ in~$(\mu\circ\nu)$ in the sense that
\begin{eqnarray}\label{intechange-equ}
(\mu\circ\nu)\sim (-1)^{\D{\aa_\ii}\D{\aa_{\ii-1}... \aa_1\bb_\mm...\bb_\jj}+\D{\bb_\jj}\D{\aa_{\ii-1}...\aa_1\bb_\mm...\bb_{\jj+1}}}
(\mu_{\aa_\ii\mapsto \bb_\jj}\circ \nu_{\bb_\jj\mapsto \aa_\ii}).
\end{eqnarray}
In particular, if $\rr=\mm$, $\aa_1=\bb_1$ and~$\D{\aa_1}=1$, then we
get~$(\mu\circ\nu)\sim -(\mu\circ\nu)$. Since~${\chart(\kk)\neq 2}$,  the term~$(\mu\circ\nu)$ can be written as a linear combination of terms that are of root numbers~$>\mr\mu+\mr\nu$ and with lengths~$\ell(\mu)+\ell(\nu)$.
\end{lemm}

\begin{proof}
By Lemmas~\ref{root-pro} and~\ref{interchange leaves}, we get
\begin{multline*}
(\mu\circ\nu) \sim  (-1)^{\D{\aa_\ii}\D{\aa_{\ii+1}... \aa_{\rr+1}}}((\aa_i\circ \mu_{\hat{\aa_\ii}})\circ \nu)
 \sim  (-1)^{|a_i||a_{i+1}... a_{r+1}|+|\mu_{\hat{\aa_\ii}}||\nu|}((a_i\circ \nu)\circ \mu_{\hat{a}_i})\\
 \sim  (-1)^{|a_i||a_{i+1}... a_{r+1}|+|\mu_{\hat{\aa_\ii}}||\nu|+|a_i||b_j... b_m|+|b_j||b_{j+1}... b_m|}
((b_j\circ \nu_{b_j\mapsto a_i})\circ \mu_{\hat{\aa_\ii}})\\
 \sim  (-1)^{|a_i||a_{i+1}... a_{r+1}|+|\mu||\nu|+|a_i||b_1... b_{j-1}|+|b_j||b_{j+1}... b_m|+|\mu_{\hat{a}_i}||\nu_{b_j\mapsto a_i}|}
((b_j\circ \mu_{\hat{\aa_\ii}})\circ \nu_{b_j\mapsto a_i})\\
 \sim (-1)^{|a_i||a_{i+1}... a_{r+1}b_1... b_{j-1}|+|\mu||\nu|+|b_j||a_{i+1}... a_{r+1}b_{j+1}... b_m|+|\mu_{\hat{a}_i}||\nu_{b_j\mapsto a_i}|}
(\mu_{a_i\mapsto b_j}\circ \nu_{b_j\mapsto a_i})\\
\sim (-1)^{|a_i||a_{i-1}... a_{1}b_m... b_{j+1}|+|b_j||a_{i}... a_{1}b_{j+1}... b_m|}
(\mu_{a_i\mapsto b_j}\circ \nu_{b_j\mapsto a_i}).
\end{multline*}
In particular, if $\rr=\mm$, $\aa_1=\bb_1$ and~$\D{\aa_1}=1$, then we obtain
\begin{multline*}
(\mu\circ\nu)\sim  (-1)^{|a_r||a_{r-1}... a_{1}|+|b_r||a_{r}... a_{1}|}
(\mu_{a_r\mapsto b_r}\circ \nu_{b_r\mapsto a_r})\\
\sim (-1)^{|a_r||a_{r-1}... a_{1}|+|b_rb_{r-1}||a_{r}... a_{1}|+|a_{r-1}||a_{r}\hat{a}_{r-1}... a_{1}|}
((\mu_{a_r\mapsto b_r})_{a_{r-1}\mapsto b_{r-1}}\circ (\nu_{b_r\mapsto a_r})_{b_{r-1}\mapsto a_{r-1}})\\
\sim  \cdots
\sim (-1)^{|a_ra_{r-1}... a_{2}|+|a_ra_{r-1}... a_{2}||a_{r}... a_{1}|+|b_rb_{r-1}... b_2||a_{r}... a_{1}|}
((a_{r+1}\circ \nu_{b_1\mapsto a_1})\circ (\mu_{\hat{a}_{r+1}})_{a_1\mapsto b_1})\\
\sim (-1)^{|a_r... a_{2}||a_{1}|+|b_r... b_2||a_{r}... a_{1}|+|a_ra_{r-1}... a_{2}b_1||b_rb_{r-1}... b_2a_1|}
((a_{r+1}\circ(\mu_{\hat{a}_{r+1}})_{a_1\mapsto b_1})\circ \nu_{b_1\mapsto a_1}).
\end{multline*}
Since~$\aa_1=\bb_1$, we obtain
$(a_{r+1}\circ(\mu_{\hat{a}_{r+1}})_{a_1\mapsto b_1})=\mu$ and~$\nu_{b_1\mapsto a_1}=\nu$.
Moreover, ${\D{\aa_1}=\D{\bb_1}=1}$ implies
that~$(-1)^{|a_r... a_{2}||a_{1}|+|b_r... b_2||a_{r}... a_{1}|+|a_ra_{r-1}... a_{2}b_1||b_rb_{r-1}... b_2a_1|}=(-1).$ Therefore, we obtain~$(\mu\circ\nu)\sim-(\mu,\nu)$.
Since $\chart(\kk)\neq 2$, the claim follows.
\end{proof}

Now we are in a position to show that the set of all the \nsuptx\ over a well-ordered set~$X=X_0 \cup X_1$ forms a linear generating set of the free GDN superalgebra~$\fns\XX$ generated by~$\XX$.
\begin{lemm}\label{generating set}
For every term~$\lambda$ in~$\XXps$, we
  have~$\lambda=\sum_i\alpha_i\lambda_i$
  for some elements~$\alpha_i$ in the field~$\kk$ and for some GDN supertableaux~$\lambda_i$ such
  that~$\ell(\lambda_i)=\ell(\lambda)$ and~$\mr{\lambda_i}\geq \mr\lambda$.
\end{lemm}
\begin{proof}
 We use induction on~$\ell(\lambda)$. For~$\ell(\lambda)\leq 2$, it is clear.
 For ${\ell(\lambda)>2}$, we use a second (downward) induction on~$\mr\lambda$.
 For~$\mr\lambda=\ell(\lambda)-1$, by Lemma~\ref{prop n-1} and by the right supercommuativity,
 we may assume~$\lambda=\lnormed{\aa_1}{\aa_{\ell(\lambda)}}$ and~$\aa_2\geq \pdots\geq\aa_{\ell(\lambda)}$. If Condition Definition~\ref{def-sup-tab}\ITEM5 is not satisfied, then~$\lambda=0$, otherwise, $\lnormed{\aa_1}{\aa_{\ell(\lambda)}}$ is already a GDN supertableau. For~$\mr\lambda<\ell(\lambda)-1$, we may assume that~$\lambda=(\mu\circ\nu)$ for some terms~$\mu,\nu$ with~${\ell(\mu)<\ell(\lambda)}$ and~$\ell(\nu)<\ell(\lambda)$.
 By induction hypothesis, both~$\mu$ and~$\nu$ can be written as linear combinations of GDN supertableaux, say~$\mu=\sum_{\ii}\alpha_\ii\mu_i'$ and~$\nu=\sum_{\jj}\beta_\jj\nu_j'$.
  Then for all~$\ii,\jj$, we have~${\mr{(\mu_\ii'\circ\nu_\jj')}=\mr{\mu_\ii'}+\max(1,\mr{\nu_\jj'})\geq \mr{\mu}+\max(1,\mr{\nu})=\mr{(\mu\circ\nu)}}$ and~${\ell((\mu_\ii'\circ\nu_\jj'))=\ell((\mu\circ\nu))}$.

Now we can assume that~$\mu$ and~$\nu$ are GDN supertableaux.
  Suppose that~${\mu=\lnormed{\aa,\mu_1}{\mu_p}}$ and~$\nu=\lnormed{\bb,\nu_1}{\nu_\qq}$,
  where all the~$\mu_\ii$ and~$\nu_\jj$ are simple terms, and~$\aa,\bb$ are elements in~$\XX$.

For~$\ell(\mu)>1$, we have~$\pp>0$ and~$$\lambda=(\mu\circ\nu)
=(-1)^{\D{\nu}\D{\mu_\pp}}
(\lnormed{\aa,\mu_1}{\mu_{\pp-1},\nu}\circ\mu_\pp).$$
By induction hypothesis again, we can write the term~$\lnormed{\aa,\mu_1}{\mu_{\pp-1},\nu}$ as a linear combination of GDN supertableaux. Therefore, we may assume that~$\ell(\mu)>1$ and~$\nu$ is a simple term. In other words, we can assume that~$\lambda=\lnormed{\aa,\lambda_1}{\lambda_n}$ $(\nn\geq 1)$, where~$\aa$ is an element in~$\XX$ and each~$\lambda_i=\rnormed{\aa_{\ii,\rr_\ii}}{\aa_{\ii,1}}$ is a simple term.
 We first show that in this case~$\lambda$ can be written as a linear combination of GDN supertableaux with the claimed conditions. This is the main case with which we should deal.

Applying the right supercommuativity whenever necessary, we may assume that the conditions of Definition~\ref{def-sup-tab}\ITEM1-\ITEM2
are satisfied. By Lemmas~\ref{root-pro}, \ref{interchange leaves} and~\ref{interchange two row leaves}, and by induction hypothesis, the conditions of Definition~\ref{def-sup-tab}\ITEM3 can also be obtained. For instance, say~$\aa_{2,2}<\aa_{4,3}$.
Then we need to interchange~$\aa_{2,2}$ and~$\aa_{4,3}$ in the sense of Equation~\eqref{intechange-equ}.  Since~$\lnormed{\aa,\lambda_1}{\lambda_n}=(-1)^{\D{\lambda_2}\D{\lambda_1}+\D{\lambda_{4}}\D{\lambda_{3}}
+\D{\lambda_{4}}\D{\lambda_{1}}}\lnormed{\aa,\lambda_2,\lambda_4,\lambda_1,\lambda_3,\lambda_5}{\lambda_n}$, we can apply Lemma~\ref{interchange two row leaves} to the term $((\aa\circ\lambda_2)\circ\lambda_4)$.

Suppose that the conditions of Definition~\ref{def-sup-tab}\ITEM4 are destroyed. If~${\aa_{i,j}=\aa_{i,j+1}\in \XX_1}$ for some integers~$\ii$ and~$\jj$ such that~$1\leq i\leq n$ and~$2\leq \jj< \rr_i$, then by left supersymmetry, we
obtain
  \begin{multline*}
  \lambda_\ii-\rnormed{\aa_{\ii,\rr_\ii}\wdots \aa_{\ii,\jj+2}, (\aa_{\ii,\jj+1}\circ\aa_{i,\jj}),\aa_{i,\jj-1} }{\aa_{i,1}}\\
  =-1(\lambda_\ii-\rnormed{\aa_{\ii,\rr_\ii}\wdots \aa_{\ii,\jj+2}, (\aa_{\ii,\jj}\circ\aa_{i,\jj+1}),\aa_{i,\jj-1} }{\aa_{i,1}}).
\end{multline*}
 Since~$\chart(\kk)\neq 2$, we have~$\lambda_\ii=\rnormed{\aa_{\ii,\rr_\ii}\wdots \aa_{\ii,\jj+2}, (\aa_{\ii,\jj+1}\circ\aa_{i,\jj}),\aa_{i,\jj-1} }{\aa_{i,1}}$ in~$\fns\XX$ and~$\mr{\rnormed{\aa_{\ii,\rr_\ii}\wdots \aa_{\ii,\jj+2}, (\aa_{\ii,\jj+1}\circ\aa_{i,\jj}),\aa_{i,\jj-1} }{\aa_{i,1}}}=2>1=\mr{\lambda_i}$. By the second induction hypothesis on root numbers, we are done.
  Therefore, we may assume that, for every integer~$\jj\geq 2$, we have~$\aa_{\ii,\jj}\neq \aa_{\ii,\jj+1}$ if~$\aa_{\ii,\jj}$ lies in~$\XX_1$.
 Similarly, we may assume that~$\aa\neq \aa_{1,\rr_1}$ if~$\aa$ lies in~$\XX_1$.
 On the other hand, if~$\aa_{\ii,2}=\aa_{\ii+1,\rr_{\ii+1}}$, then
 we have~\begin{multline*}
   \lambda=\alpha_1\lnormed{\aa,\lambda_\ii,\lambda_{\ii+1}, \lambda_1,\wdots, \lambda_{\ii-1},\lambda_{\ii+2}}{\lambda_n}\\
   \sim\alpha_2\lnormed{\aa_{\ii,2},(\lambda_\ii)_{\aa_{\ii,2}\mapsto\aa},\lambda_{\ii+1}, \lambda_1,\wdots, \lambda_{\ii-1},\lambda_{\ii+2}}{\lambda_n}\\
   \sim\alpha_3\lnormed{\aa_{\ii,2},\lambda_{\ii+1}, (\lambda_\ii)_{\aa_{\ii,2}\mapsto\aa},\lambda_1,\wdots, \lambda_{\ii-1},\lambda_{\ii+2}}{\lambda_n}
 \end{multline*}
 for some elements~$\alpha_1,\alpha_2$ and~$\alpha_3$ in~$\kk$. The claim follows by the above reasoning.

 Finally, by Lemma~\ref{interchange two row leaves}, right supercommutativity and by induction hypothesis,  the conditions of Definition~\ref{def-sup-tab}\ITEM5 can also be satisfied.

For~$\ell(\mu)=1$, we have~$\ell(\nu)\geq 2$ and thus~$\qq\geq 1$.
For~$\qq=1$, the term~$\lambda=(\mu\circ\nu)$ is a simple term. By Lemma~\ref{interchange leaves} and induction hypothesis on root number, we are done. For~$\qq>1$, we shall resort to the case of~$\ell(\mu)>1$. By left supersymmetry, we have
\begin{multline*}
  \lambda=(\mu\circ\lnormed{\bb,\nu_1}{\nu_q})
  =((\mu\circ\lnormed{b,\nu_1}{\nu_{q-1}})\circ\nu_q)\\
  +(-1)^{\D{\mu}\D{\lnormed{b,\nu_1}{\nu_{q-1}}}}
  ((\lnormed{b,\nu_1}{\nu_{q-1}} \circ(\mu\circ\nu_q))
 -((\lnormed{b,\nu_1}{\nu_{q-1}} \circ\mu)\circ\nu_q)).
\end{multline*}
By induction hypothesis and the above reasoning for the case of~$\ell(\mu)>1$, the result follows.
\end{proof}

\section{A linear basis of~$\fns\XX$ and a Poincar\'{e}-Birkhoff-Witt type Theorem}\label{basis-and-pbw-sec}
Our aim in this section is to show that the set~$\tab$ of all the GDN supertableaux over a well-ordered set~$\XX=\XX_0\cup\XX_1$ forms a linear basis of the free GDN superalgebra~$\fns\XX$. We already know that it is a linear generating set of~$\fns\XX$, so what remains is to prove the linear independence. We shall also prove a
PBW type theorem for GDN superalgebra, that is, every GDN superalgebra can be embedded into its universal enveloping associative differential supercommutative algebra.

\subsection{Associative differential supercommutative algebra}In this subsection, we shall first construct the free associative differential supercommuative algebra generated by~$\XX$. It will be instrumental in proving the linear independence of the set~$\tab$ of all the GDN supertableaux over~$\XX$.

Recall that a supercommutative algebra is a superalgebra~$\mA$ satisfying the following identity:
$$\xx\cdot\yy=(-1)^{\D\xx\D\yy}\yy\cdot\xx $$
 for all elements $\xx,\yy$ in~$\mathcal{A}_{0}\cup\mathcal{A}_{1}$, and
 an \textit{associative differential supercommutative algebra} is an associative supercommutative algebra~$(\mA, \cdot,\oD)$ with a linear derivation~$D$ of parity~$0$ satisfying that~$D(\mA_i)\subseteq \mA_i$ $(\ii=0,1)$ and the
identity:
$$ \oD(\xx\cdot \yy)=\oD(\xx)\cdot \yy+\xx \cdot \oD(\yy),$$
 for all elements $\xx,\yy$ in~$\mathcal{A}_{0}\cup\mathcal{A}_{1}$.

S.I. Gelfand~\cite{Gelfand} pointed out that, every \da~$(\mA,\cdot,\oD)$ becomes a GDN algebra under the new operation~$\circ$ defined by~$\xx\circ\yy:=\xx\cdot \oD(\yy)$, and with the help of this discovery, Dzhumadildaev and L\"{o}fwall proved that the set of all the GDN tableaux over a well-ordered set~$\XX$ forms a linear basis of the free GDN algebra generated by~$\XX$. This idea motivates us to establish the connection of GDN superalgebra and \adsa. The proof for the following observation is straightforward and thus omitted.

\begin{lemm}\label{become-gdn}
 For every \adsa~$(\mA,\cdot,\oD)$, if we define a new bilinear operation on~$\mA$ by the rule:
 $$\xx\circ\yy=\xx\cdot\oD(\yy)$$
 for all elements~$\xx$ and~$\yy$ in~$\mA$,
 then~$(\mA,\circ)$ becomes a GDN superalgebra.
\end{lemm}

Let~$\mA$ be an associative differential supercommutative algebra over~$\kk$ generated by a set~$\XX=\XX_0\cup\XX_1$. We say~$\mA$ is \emph{free} on~$\XX$ if, for every map~$\psi$ of~$\XX$ into an associative differential supercommuative algebra~$\mB=\mB_0\oplus\mB_1$ such that~$\psi(\XX_i)\subseteq\mB_i$ $(i=0,1)$, there exists a unique homomorphism~$\varphi:$ $\mA\rightarrow \mB$ extending~$\psi$. We shall construct the free associative differential supercommutative algebra~$\fdsa$ generated by a set~$\XX$ directly.

Define~$\oD^{0}(\aa)=\aa$ for every~$\aa$ in~$\XX$. Define~$\YY=\{\oD^{\nn}(a)\mid \aa\in \XX, \nn\geq 0, \nn\in \mathbb{N}\}$ and let~$\YYplus$ be the free semigroup (without unit) generated by~$\YY$. For every~$\uu=\oD^{\ii_1}(\aa_1)...\oD^{\ii_n}(\aa_n)$ in~$\YYplus$, define the parity~$\D\uu$ of~$\uu$ to be~$\D{\aa_1}+\pdots+\D{\aa_n}$ modulo by~$2$, and define~$\oD^{\ii}(\aa)<\oD^\jj(\bb)$ if~$(\ii,\aa)<(\jj,\bb)$ lexicographically. Finally, define
\begin{gather*}
\parbox{152mm}{$\DX:=\{\oD^{\ii_1}(\aa_1)...\oD^{\ii_n}(\aa_n)\in \YYplus\mid \oD^{\ii_1}(\aa_1)\wdots\oD^{\ii_n}(\aa_n)\in Y$, $\oD^{\ii_1}(\aa_1)\leq \pdots\leq \oD^{\ii_n}(\aa_n)$, if~$\aa_\pp=\aa_\qq\in \XX_1$  for some integers~$\pp\neq \qq\leq n$,  then ~$\ii_\pp\neq\ii_\qq\}$.}
\end{gather*}

Let~$\kk\DX$ be the $\kk$~linear space with a $\kk$-basis~$\DX$.
Define a bilinear operation~$\cdot$ on the space~$\kk\DX$ as follows: For all~
\begin{equation}\label{xyform}
 \uu=\oD^{\ii_1}(\aa_1)...\oD^{\ii_n}(\aa_n), \ \vv=\oD^{\jj_1}(\bb_1)...\oD^{\jj_m}(\bb_m )  \mbox{ and } \oD^{\jj}(\bb)
 \mbox{ in } \DX,
\end{equation}
if~$\bb$ lies in~$\XX_1$ and~$\oD^{\jj}(\bb)=\oD^{\ii_t}(\aa_t)$ for some integer~$\tt\leq n$, then~${\uu\cdot\oD^{\jj}(\bb)}$ is defined to be~$0$.
Otherwise, assume that~$\oD^{\ii_1}(\aa_1)...\oD^{\ii_{t-1}}(\aa_{t-1})\oD^{\jj}(\bb)
\oD^{\ii_{t}}(\aa_{t})... \oD^{\ii_n}(\aa_n)$ lies in~$\DX$ for some integer~$\tt$ satisfying~$1\leq \tt\leq \nn+1$, where~$\tt=1$ (or~$\tt=n+1$, resp.) means~$\oD^{\ii_1}(\aa_1)...\oD^{\ii_{t-1}}(\aa_{t-1})$ (or~$\oD^{\ii_{t}}(\aa_{t})... \oD^{\ii_n}(\aa_n)$, resp.) is an empty sequence. Then define~$\uu\cdot\oD^{\jj}(\bb)$ to be~$$(-1)^{\sum_{\pp}(\D{\aa_{\pp}}\D\bb)}
\oD^{\ii_1}(\aa_1)...\oD^{\ii_{t-1}}(\aa_{t-1})\oD^{\jj}(\bb)
\oD^{\ii_{t}}(\aa_{t})... \oD^{\ii_n}(\aa_n),$$
where the sum is over all the integer~$\pp$ such that~$\oD^{\ii_p}(\aa_p)>\oD^{\jj}(\bb)$.
Next, the product~$\uu\cdot\vv$ is defined inductively as follows: $$\uu\cdot\vv
:=(\uu\cdot\oD^{\jj_1}(\bb_1))\cdot\oD^{\jj_2}(\bb_2)...\oD^{\jj_m}(\bb_m).
 $$ Finally, define a unary linear operation~$\oD$ on~$\kk\DX$ as follows:
  $$\oD(\uu)=\sum_{1\leq t\leq n}
  (\oD^{\ii_1}(\aa_1)...\oD^{\ii_{t-1}}(\aa_{t-1}) \cdot \oD^{\ii_{t}+1}(\aa_{t})) \cdot \oD^{\ii_{t+1}}(\aa_{t+1})...\oD^{\ii_{n}}(\aa_{n}).$$

The following lemma offers an explicit formula for calculating the product of arbitrary two elements in~$\DX$.
\begin{lemm}\label{product-fomula}
Let~$\uu$ and~$\vv$ be as in Equation~\eqref{xyform}. If~$\uu\cdot\vv\neq 0$, then we have~$$\uu\cdot \vv=(-1)^{\sum_{(\pp,\qq)}(\D{\aa_\pp}\D{\bb_\qq})}\oD^{l_1}(\dd_1)...\oD^{l_{n+m}}
(\dd_{n+m}),$$
where~$\oD^{l_1}(\dd_1)\wdots\oD^{l_{n+m}}(\dd_{n+m})$ is a reordering of~$\oD^{\ii_1}(\aa_1)\wdots\oD^{\ii_n}(\aa_n)$, $\oD^{\jj_1}(\bb_1)\wdots\oD^{\jj_m}(\bb_m)$ such that~$\oD^{l_1}(\dd_1)...\oD^{l_{n+m}}(\dd_{n+m})$ lies in~$\DX$, and the sum is over all the pairs~$(\pp,\qq)$ such that~$\oD^{\ii_p}(\aa_p)>\oD^{\jj_\qq}(\bb_\qq)$.   Moreover, the equality~${\uu\cdot \vv =0}$ holds if and only if, for some integers~$t\leq n$ and~$l\leq m$, we have~$\ii_t=\jj_l$ and $\aa_t=\bb_l\in \XX_1$.
\end{lemm}
\begin{proof}
The second claim is clear, so we just prove the first one. Use induction on~$\mm$. For~$\mm=1$, the claim follows by the definition of the operation~$\cdot$. For~$\mm>1$, since the inequality~$
\oD^{\jj_1}(\bb_1)\leq\oD^{\jj_2}(\bb_2)\leq \pdots \leq \oD^{\jj_m}(\bb_m)$ holds, we obtain
\begin{multline*}
 \uu\cdot\vv=(\uu\cdot \oD^{\jj_1}(\bb_1))\cdot\oD^{\jj_2}(\bb_2)...\oD^{\jj_m}(\bb_m)\\
 =(-1)^{\sum_{\pp}(\D{\aa_{\pp}}\D{\bb_1})}
\oD^{\ii_1}(\aa_1)...\oD^{\ii_{t-1}}(\aa_{t-1})\oD^{\jj_1}(\bb_1)
\oD^{\ii_{t}}(\aa_{t})... \oD^{\ii_n}(\aa_n) \cdot \oD^{\jj_2}(\bb_2)...\oD^{\jj_m}(\bb_m)\\
=(-1)^{\sum_{(\pp,\qq)}(\D{\aa_\pp}\D{\bb_\qq})}\oD^{l_1}(\dd_1)...\oD^{l_{n+m}}
(\dd_{n+m})
\end{multline*}
with the desired properties.
\end{proof}
Now we are in a position to show that, endowed with the defined operations~$\cdot$ and~$\oD$, the vector space~$\kk\DX$ becomes a free \adsa.
\begin{lemm}\label{free-dsa}
The algebra~$(\kk\DX,\cdot,\oD)$ is isomorphic to the free \adsa~$\fdsa$ generated by~$\XX$. In particular, if we define a linear operation~$\circ$ on~$(\kk\DX,\cdot,\oD)$ by the rule: $\uu\circ\vv=\uu\cdot \oD(\vv)$ for all~$\uu$ and~$\vv$ in~$\DX$, then~$(\kk\DX,\circ)$ becomes a GDN superalgebra.
\end{lemm}
\begin{proof}We first show that~$(\kk\DX,\cdot,\oD)$ is an \adsa.
By Lemma~\ref{product-fomula}, the associativity is straightforward.
As for the supercommutativity, let~$\uu$ and~$\vv$ be as in Equation~\eqref{xyform}. For~$\uu\cdot\vv=0$, it is clear that~$\vv\cdot\uu=0=(-1)^{\D\uu\D\vv}\uu\cdot\vv$.
For~$\uu\cdot\vv\neq 0$,
with the same notation of Lemma~\ref{product-fomula}, we have
$$
\uu\cdot\vv=(-1)^{\sum_{(\pp,\qq)}(\D{\aa_\pp}\D{\bb_\qq})}\oD^{l_1}(\dd_1)...\oD^{l_{n+m}}(\dd_{n+m})
,
$$
where the sum is over all the pairs~$(\pp,q)$ such that~$\oD^{\ii_p}(\aa_p)>\oD^{\jj}(\bb)$.
Similarly, we get
$$
\vv\cdot\uu=(-1)^{\sum_{(\pp',\qq)}(\D{\aa_{\pp'}}\D{\bb_\qq})}\oD^{l_1}(\dd_1)...\oD^{l_{n+m}}
(\dd_{n+m}),
$$
where the sum is over all the pairs~$(\pp',q)$ such that~$\oD^{\ii_{p'}}(\aa_{p'})<\oD^{\jj}(\bb)$.
Combining the above two formulas, we get
$$\uu\cdot \vv=(-1)^{\sum_{(\pp,\qq)}(\D{\aa_\pp}\D{\bb_\qq})}\vv\cdot \uu,$$
where the sum is over all the pairs~$(\pp,q)$ such that~$\oD^{\ii_p}(\aa_p)\neq\oD^{\jj}(\bb)$.
Moreover, if~$\ii_p=\jj_q$ and~$\aa_p=\bb_q$ for some integers~$\pp,\qq$ such that~$1\leq p\leq n$ and~$1\leq q\leq m$, then~$\aa_\pp$ lies in~$\XX_0$ and thus~$(-1)^{\D{\aa_p}\D{\bb_q}}=1$. So we obtain
$\vv\cdot\uu=(-1)^{\D\uu\D\vv}\uu\cdot\vv$.

To show that~$\oD(\uu\cdot\vv)=\oD(\uu)\cdot \vv+\uu\cdot\oD(\vv)$, we use induction on~$\mm$. For~$\mm=1$, we have
\begin{multline*}
  \oD(\uu\cdot\oD^{\jj_1}(\bb_1))
  =(-1)^{\sum_\pp(\D{\aa_\pp}\D{\bb_1})}
  \oD(\oD^{\ii_1}(\aa_1)...\oD^{\ii_{t-1}}(\aa_{t-1})\oD^{\jj}(\bb)
\oD^{\ii_{t}}(\aa_{t})... \oD^{\ii_n}(\aa_n))\\
=(-1)^{\sum_\pp(\D{\aa_\pp}\D{\bb_1})}
((\oD^{\ii_1}(\aa_1)...\oD^{\ii_{\tt-1}}(\aa_{\tt-1}) \cdot \oD^{\jj_1+1}(\bb_{1})) \cdot \oD^{\ii_{\tt}}(\aa_{\tt})...\oD^{\ii_{n}}(\aa_{n})\\
+\sum_{1\leq \qq\leq \tt-1}(\oD^{\ii_1}(\aa_1)...\oD^{\ii_{q-1}}(\aa_{q-1}) \cdot \oD^{\ii_{q}+1}(\aa_{q})) \cdot \oD^{\ii_{q+1}}(\aa_{q+1})...\oD^{\jj_1}(\bb_1)...\oD^{\ii_{n}}(\aa_{n})\\
+
\sum_{\tt\leq \qq\leq \nn}(\oD^{\ii_1}(\aa_1)...\oD^{\jj_1}(\bb_1)...\oD^{\ii_{p-1}}(\aa_{p-1}) \cdot \oD^{\ii_{q}+1}(\aa_{q})) \cdot \oD^{\ii_{q+1}}(\aa_{q+1})...\oD^{\ii_{n}}(\aa_{n})
)\\
=  \oD(\uu) \cdot \oD^{\jj_1}(\bb_{1})
+ \uu \cdot \oD^{\jj_1+1}(\bb_{1})  \mbox{ (by applying associativity and supercommutativity)}.
\end{multline*}
 For~$\mm>1$, we obtain
 \begin{multline*}
   \oD(\uu\cdot \vv)=\oD((\uu\cdot\oD^{\jj_1}(\bb_1))
   \cdot\oD^{\jj_2}(\bb_2)...\oD^{\jj_m}(\bb_m))\\
   =\oD(\uu\cdot\oD^{\jj_1}(\bb_1))
   \cdot\oD^{\jj_2}(\bb_2)...\oD^{\jj_m}(\bb_m)+  (\uu\cdot\oD^{\jj_1}(\bb_1))
   \cdot\oD(\oD^{\jj_2}(\bb_2)...\oD^{\jj_m}(\bb_m))\\
   =(\oD(\uu)\cdot\oD^{\jj_1}(\bb_1))
   \cdot\oD^{\jj_2}(\bb_2)...\oD^{\jj_m}(\bb_m)
   + (\uu\cdot \oD^{\jj_1+1}(\bb_1))
   \cdot \oD^{\jj_2}(\bb_2)...\oD^{\jj_m}(\bb_m)\\ +  (\uu\cdot\oD^{\jj_1}(\bb_1))
   \cdot\oD(\oD^{\jj_2}(\bb_2)...\oD^{\jj_m}(\bb_m))
   =\oD(\uu)\cdot \vv+\uu\cdot\oD(\vv).
 \end{multline*}
 Therefore, $(\kk\DX,\cdot,\oD)$ is an \adsa.

It remains to show that~$(\kk\DX,\cdot,\oD)$ is free on~$\XX$. By applying associativity and supercommutativity in~$\fdsa$,
it is easy to see that the set of all the monomials of the form:
$$(((...(\oD^{\ii_1}(\aa_1) \cdot \oD^{\ii_2}(\aa_2))\cdot\pdots)
\cdot\oD^{\ii_{n-1}}(\aa_{n-1}))
\cdot\oD^{\ii_{n}}(\aa_{n}) ) \mbox{ (left-normed bracketting)}$$ such that~$\oD^{\ii_1}(\aa_1)...\oD^{\ii_n}(\aa_n)$ lies in~$\DX$ forms a linear generating set of~$\fdsa$.
Define a map~$\psi$: $\XX \rightarrow \kk\DX$ by~$\psi(\aa)=\aa$ for every~$\aa$ in~$\XX$, and extend~$\psi$ to a superalgebra homomorphism~$\tl{\psi}$: $ \fdsa \rightarrow \kk\DX$. Then
$$\tl{\psi}((((...(\oD^{\ii_1}(\aa_1) \cdot \oD^{\ii_2}(\aa_2))\cdot\pdots)
\cdot\oD^{\ii_{n-1}}(\aa_{n-1}))
\cdot\oD^{\ii_{n}}(\aa_{n})))
=\oD^{\ii_1}(\aa_1)...\oD^{\ii_n}(\aa_n).$$ Since the set~$\DX$ is linearly independent in~$\kk\DX$, the homomorphism~$\tl\psi$ is an isomorphism.
\end{proof}
Thanks to Lemma~\ref{free-dsa}, we can identify~$\fdsa$ with~$\kk\DX$.

\subsection{The linear independence of the set~$\tab$}
 Our aim in this subsection is to show that the set of all the GDN supertableaux~$\tab$ over~$\XX$ is linearly independent. Our strategy is to construct an GDN superalgebra homomorphism
 from~$(\fns\XX,\circ)$ to~$(\kk\DX, \circ)$, and show that the image of~$\tab$ is linearly independent in~$\kk\DX$, where the operation~$\circ$ is defined in Lemma~\ref{become-gdn}.

We define an ordering~$<$ on~$\DX$ as follows:
For all~$\uu$ and~$\vv$ be as in Equation~\eqref{xyform}, we define
\begin{gather}\label{ordering}\tag{$\ast$}
\parbox{132mm}{$\uu<\vv \Leftrightarrow (n,\ii_n, \aa_n\wdots \ii_1,\aa_1)<(m,\jj_m, \bb_m, \wdots
\jj_1,\bb_1) \mbox{ lexicographically}$,}
\end{gather}
and define the \emph{length}~$\ell(\uu)$ of~$\uu$ to be~$\nn$.

For every element~$\ff=\sum_{1\leq i\leq n}\alpha_i\uu_i$ with each~$\alpha_i\neq 0$ in~$\kk$ and~$\uu_1>\uu_2>...>\uu_n$ in~$\DX$, we call~$\ffb:=\uu_1$ the \emph{leading monomial} of~$\ff$, and call~$\lcoe\ff:=\alpha_1$ the \emph{leading coefficient} of~$\ff$.

Now we are ready to show that the set~$\tab$ is linearly independent in~$\fns\XX$.  Recall that by Lemma~\ref{become-gdn}, $(\kk\DX, \circ)$ is a GDN superalgebra.
  \begin{theorem}\label{sup-basis}
  Let~$\varphi$: $(\fns\XX, \circ)\rightarrow (\kk\DX,\circ)$ be a GDN superalgebra homomorphism induced by~$\varphi(\aa)=\aa$ for every element~$\aa$ in~$\XX$. Then~$\varphi$ is injective. Moreover, the set~$\tab$ of all the GDN supertableaux over $\XX$ forms a linear basis of the free GDN superalgebra~$\fns\XX$.
\end{theorem}

\begin{proof}
We first show that~$\varphi$ is injective.
Let~$\mu$ be a GDN supertableau as in Equation~\eqref{form-sup-tab}.
Then it is easy to see that~
$$\varphi(\mu)=\varphi(\aa)\cdot \oD(\varphi(\mu_1))\cdot \pdots \cdot \oD(\varphi(\mu_n)).$$
Therefore, it is straightforward to show that
$$
\overline{\varphi(\mu)}=a_{n,2}...a_{n,r_{n}}a_{n-1,2}...a_{n-1,r_{n-1}}...a_{1,2}...\aa_{1,r_{1}}\aa \oD^{r_n}(\aa_{n,1})...\oD^{r_1}(\aa_{1,1}),
$$
and~$\lcoe{\varphi(\mu)}=1$ or~$\lcoe{\varphi(\mu)}=-1$.

Therefore, for all  GDN supertableaux $\mu$ and~$\nu$, if~$\mu\neq \nu$, then we obtain~$\ov{\varphi(\mu)}\neq \ov{\varphi(\nu)}$.
 Suppose that for some pairwise different GDN supertableaux~$\mu_1\wdots\mu_n$ in~$\tab$, for some nonzero elements~$\alpha_1\wdots\alpha_n$ in~$\kk$, we have  $\sum\alpha_i\mu_i=0$.
 Then the equality~$\sum\varphi(\alpha_i\mu_i)=0$ contradicts to the fact that $\ov{\varphi(\mu_i)}$  are pairwise different. Therefore, the set~$\tab$ is linearly independent and the homorphism~$\varphi$ is injective.
 In particular, by Lemma~\ref{generating set}, the set~$\tab$ is a linear basis of~$\fns\XX$.
\end{proof}

\subsection{A Poincar\'{e}-Birkhoff-Witt type Theorem}
We call an \adsa~$\mB=\mB_0\oplus \mB_1$ a universal enveloping algebra of a GDN superalgebra~$\mA=\mA_0\oplus \mA_1$ if, there is a linear map~$\psi$: $\mA\rightarrow \mB$ satisfying $\varphi(\mA_i)\subseteq\mB_i$ ($i=0,1$) and
\begin{equation}\label{universal-def}
  \psi(\xx\circ\yy)=\psi(\xx)\cdot \oD(\psi(\yy))
\end{equation}
for all~$\xx$ and~$\yy$ in~$\mA$, and the following holds: for an arbitrary \adsa~$\mathcal{C}=\mathcal{C}_0\oplus\mathcal{C}_1$,
for every linear map~$\psi'$: $\mA\rightarrow \mathcal{C}$ satisfying the equation~$\psi'(\xx\circ\yy)=\psi'(\xx)\cdot \oD(\psi'(\yy))$ for all~$\xx$ and~$\yy$ in~$\mA$,   and~$\psi'(\mA_i)\subseteq\mathcal{C}_i$ ($i=1,2$), there exists a unique homomorphism of \adsa s~${\varphi:\ \mB\rightarrow\mathcal{C}}$ such that~$\varphi \circ \psi=\psi'$.
It is easy to see that whenever such an universal enveloping algebra~$\mB$ exists, then it is unique up to isomorphism.

Let $\mA=\mA_0\oplus \mA_1$  be a  superalgebra and let~$\SS$ be a subset of~$\mA$. We call~$S$ a \textit{homogeneous} set if~$\SS$ is a subset of~$\mA_0\cup \mA_1$.
 For every homogenous subset~$\SS$ of~$\fns\XX$, the notation~$\ns\XX\SS$ means the quotient superalgebra~$\fns\XX/\Id\SS$, where~$\Id\SS$ means the ideal of~$\fns\XX$ generated by~$\SS$.
  Let~$\varphi$ be as that in Theorem~\ref{sup-basis}, and denote by~$\IdD{\varphi{(\SS)}}$ the \adsa\ ideal of~$(\kk\DX, \cdot,\oD)$ generated by~$\varphi(\SS)$.
 By convention, the notation~$\kds{\varphi{(\SS)}}$ means  the \adsa\ generated by~$X$ with the set~$\varphi{(\SS)}$ of defining relations, that is, the quotient superalgebra~$\kk\DX/\IdD{\varphi{(\SS)}}$. Then it is easy to see that, for every GDN superalgebra~$\ns\XX\SS$, the \adsa~$\kds{\varphi(\SS)}$ is the universal enveloping algebra of~$\ns\XX\SS$.

Our aim in this subsection is to show that every GDN superalgebra can be embedded into its universal enveloping \adsa. We first consider the subalgebra of~$(\kk\DX,\circ)$ (as GDN superalgebra) generated by~$\XX$.

For every monomial~$\uu=\oD^{\ii_1}(\aa_1)...\oD^{\ii_n}(\aa_n)$ in~$\DX$, define the \emph{weight}~$\wt(\uu)$ of~$\uu$ to be~${(\sum_{1\leq j\leq n}\ii_\jj)-n
+1}$. Then it is easy to see that~$\wt(\uu\cdot \vv)=\wt(\uu)+\wt(\vv)-1$ for all~$\uu$ and~$\vv$ in~$\DX$ such that~$\uu\cdot \vv\neq 0$.

The following lemma offers another linear basis of the free GDN superalgebra generated by~$\XX$, that is, the set~$\DDX$ of all the monomials of weight~$0$ in~$\DX$.
 \begin{lemm}\label{isomorphism}
 Let~$\kk\DDX$ be the subspace of~$\kk\DX$ spanned by all the monomials of weight~$0$ in~$\DX$. Then~$(\kk\DDX,\circ)$ is the subalgebra of~$(\kk\DX,\circ)$ generated by~$\XX$.  Moreover, let~$\varphi  :\fns\XX\longrightarrow \kk\DDX$ be the GDN superalgebra homomorphism induced by $\varphi(a)=a$ for every~$\aa$ in~$\XX$.  Then~$\varphi$ is an isomorphism.
 \end{lemm}

\begin{proof}
We first show that~$(\kk\DDX,\circ)$ is a GDN superalgebra. It is enough to show that, for all~$\uu$ and~$\vv$ in~$\DDX$, the product~$\uu\cdot \oD(\vv)$ lies in~$\kk\DDX$. Assume that~$\uu$ and~$\vv$ are as in Equation~\ref{xyform} such that~$\uu\cdot \oD\vv\neq 0$ and~$\wt(\uu)=\wt(\vv)=0$. Then by Lemma~\ref{product-fomula} and by the definition of the operation~$\oD$, we obtain that each monomial in~$\uu\cdot\oD(\vv)$ is of weight~
$$\sum_{1\leq t\leq n}i_t+\sum_{1\leq l\leq m}\jj_l+1-n-m+1
=(\sum_{1\leq t\leq n}i_t-n+1)+(\sum_{1\leq l\leq m}\jj_l-m+1)=0.$$

To show that every monomial~$\uu$ of weight~$0$ lies in the subalgebra of~$(\kk\DX,\circ)$ generated by~$\XX$, we use induction on~$\uu$ with respect to the order~$<$ defined by~\eqref{ordering}. For~$\uu=\aa\in \XX$, it is obvious. For~$\uu=\oD^{j_1}(\bb_1)...\oD^{j_m}(\bb_m)$ in~$\DX$ such that~$(\sum_{1\leq l\leq m}j_l)-m+1=0$, we have~$\jj_1=0$, because~$(\jj_1,\bb_1)\leq \pdots\leq(\jj_m,\bb_m)$ lexicographically forces~$\jj_1\leq \pdots\leq \jj_m$.
Therefore, we may assume that~$$\uu=a_{n,2}...a_{n,r_{n}}a_{n-1,2}...a_{n-1,r_{n-1}}...a_{1,2}...\aa_{1,r_{1}}\aa \oD^{r_n}(\aa_{n,1})...\oD^{r_1}(\aa_{1,1})\in \DX,$$
where~$1\leq \rr_\nn\leq \pdots\leq \rr_1$ and~$\nn\geq 1$.
Let~$\mu$ be a GDN supertableau as in Equation~\eqref{form-sup-tab}. Then~$\mu$ lies in the subalgebra of~$(\kk\DX,\circ)$ generated by~$\XX$ and it is straightforward to show that~$\ov\mu=\uu$. By induction hypothesis, the element~$\uu-\lcoe{\mu}^{-1}\mu$ lies in the subalgebra of~$(\kk\DX,\circ)$ generated by~$\XX$. The first claim of the lemma follows.

As for the second claim, notice that by the proof of Theorem~\ref{sup-basis}, the homomorphism~$\varphi$ is an injection. By the first claim, the homomorphism~$\varphi$ is an epimorphism.
\end{proof}

By Lemma~\ref{isomorphism}, we can identify~$(\kk\DDX,\circ)$ with~$(\fns\XX, \circ)$. This identification indicates some properties inherited
from~$\kk\DDX$. For instance, the following corollary offers a sufficient condition under which~$\fns\XX$ is nilpotent.
\begin{coro}
If $\XX=\XX_1$ is a finite set, where  each element of~$\XX_1$ is of parity~$1$, then~$\fns\XX$ is nilpotent. In particular,
 for every GDN superalgebra~$\mA=\mA_0\oplus\mA_1$, if~$\mA$ is generated by finite elements of~$\mA_1$,
 then~$\mA$ is nilpotent.
\end{coro}

\begin{proof} It is enough to show that, there is some positive integer~$\nn$ such that for every GDN supertableau~$\mu$, the inequality~$\ell(\mu)\geq \nn$ implies that~$\mu=0$.
Let~$\varphi$ be as in Lemma~\ref{isomorphism}. For every \nsupt~$\mu$,
we have $\varphi(\mu)=\sum_{1\leq p\leq \qq} \alpha_p \uu_p$ for some nonzero elements~$\alpha_i$ in~$\kk$, and for some monomials~$\uu_p$ in~$\DX$ such that~$\wt(\uu_p)=0$ and~$\ell(\uu_p)=\ell(\mu)$.
Say~$\uu=\uu_p$ for some integer~$\pp\leq \qq$. Then we may assume that~$$\uu=c_1... c_m\oD(b_1)... \oD(b_t)\oD^{r_1}(a_1)...\oD^{r_n}(a_n)$$ for some elements~$a_i,b_j, c_l$ in~$\XX_1$
such that~$2\leq r_1\leq\pdots \leq r_n$. Then we have~$$n\leq (r_1-1)+ (r_2-1)+ \pdots + (r_n-1)=m-1.$$
 Therefore, we obtain~$\ell(\mu)=\nn+\mm+\tt \leq 2\mm+\tt-1$.
So if~$\ell(\mu)>3(\sharp\XX_1)$, where~$\sharp\XX_1$ is the cardinality of~$\XX_1$, then~$t>(\sharp\XX_1)$ or~$m>(\sharp\XX_1)$,  both of which imply that $\varphi(\mu)=0$.
Since~$\varphi$ is an isomorphism, we get~$\mu=0$.
\end{proof}

Let~$\SS$ be a homogeneous subset of~$\fns\XX$ and let~$\varphi$ be as that in Lemma~\ref{isomorphism}. Denote by~$\IDD{\varphi(\SS)}$ the GDN superalgebra ideal of~$(\kk\DDX,\circ)$ generated by~$\varphi(\SS)$ and denote by~$\kdds{\varphi(\SS)}$  the quotient of~$(\kk\DDX,\circ)$ and~$\IDD{\varphi(\SS)}$.

By Lemma~\ref{isomorphism}, it is clear that~$\kdds{\varphi(\SS)}$ is isomorphic to~$\ns\XX\SS$. Therefore, for the embedding, it is enough to prove that~$(\kdds{\varphi{(\SS)}},\circ)$ can always be embedded into~$(\kds{\varphi(\SS)},\circ)$. We shall first investigate the elements of~$\IdD{\varphi(\SS)}$ and investigate those of~$\IDD{\varphi(\SS)}$.

Since~$\SS$ is homogeneous, it is easy to see that~$\varphi(\SS)$ is also homogeneous, in particular, for every element~$\ss$ in $\SS$, the parity~$\D{\varphi(\ss)}$ of~$\varphi(\ss)$ is well-defined.  Therefore, for every monomial~$\uu$ in~$\DX$, we have~$\uu\cdot \varphi(\ss)=(-1)^{\D{\varphi{(\ss)}}\D\uu}\varphi{(\ss)}\cdot\uu$. By applying right supercommutativity,   we obtain
$$\IdD{\varphi(\SS)}=\mathsf{span}_{\kk}\{\uu\cdot \oD^{\tt}(\varphi{(\ss))} \mid u\in \DX,   t\in \mathbb{Z}_{\geq 0},   s\in S\},$$
where~$\oD^0(\varphi(\ss))$ is defined to be~$\varphi(\ss)$.
We are now ready to describe the ideal of~$(\kk\DDX,\circ)$ generated by the set~$\varphi(\SS)$.
\begin{lemm}\label{Novikov ideal}
Let~$\SS$ be a homogeneous subset of~$\fns\XX$ and let~$\varphi$ be as in Lemma~\ref{isomorphism}.  Suppose that $\IDD{\varphi(\SS)}$ is the  ideal of the GDN superalgebra~$(\kk\DDX,\circ)$ generated by $\varphi(\SS)$. Then we have
 \begin{equation}\label{ids-form}
\IDD{\varphi(\SS)}=\mathsf{span}_{k}\{\uu\cdot \oD^{\tt}(\varphi{(\ss))} \mid u\in \DX, \  t\in \mathbb{Z}_{\geq 0}, \   s\in S, \   \wt(\overline{\uu\cdot \oD^{\tt}(\varphi{(\ss))}})=0\}.
 \end{equation}
\end{lemm}

\begin{proof}
Since~$\SS$ is homogeneous, it is clear that the right part of Equation~\ref{ids-form} is an ideal including~$\varphi(\SS)$. So to prove the lemma, it is enough to show that~$\uu\cdot \oD^{\tt}(\varphi{(\ss))}$ lies in~$\IDD{\varphi(\SS)}$
whenever~$\wt(\overline{\uu\cdot \oD^{\tt}(\varphi{(\ss))}})=0$. Since every monomial in the expansion of~$\oD^{\tt}(\varphi(\ss))$ has weight~$\tt$, we may suppose that~$$\uu=a_{1}...  a_{m}b_{1}...b_{t}D^{r_1}(c_{1})...D^{r_n}(c_{n})$$
lies in~$\DX$ such that~$m=r_1+\dots +r_n-n$ and $ r_n\geq r_{n-1}\geq \dots \geq r_1\geq 1$.
 Then in~$\kk\DDX$, we have
$$ \uu\cdot \oD^{\tt}(\varphi{(\ss))}=\alpha  \oD^{t}(s) \cdot \bb_{1}...b_{t}
\cdot \aa_{1}...\aa_{m}\oD^{r_1}(\cc_{1})...\oD^{r_n}(\cc_{n}) $$ for some integer~$\alpha$.
So the lemma will be clear if we show that the following two claims hold:

\ITEM1  The polynomial~$D^{t}(\varphi(\ss))\cdot b_{1}... b_{t}$ lies in~$\IDD{\varphi(\SS)}$ if~$\ss$ lies in $\SS$;

\ITEM2  The polynomial~$f\cdot a_{1}...a_{r-1}\cdot D^{r}(c) $ lies in~$\IDD{\varphi(\SS)}$ if $\ff$ is a homogeneous polynomial in~$\IDD{\varphi(\SS)}$.

 To prove~\ITEM1, we use induction on $t$. For~$t=0$, we get~$D^{t}(\varphi(\ss))\cdot b_{1}... b_{t}=\varphi(\ss)\in \IDD{\varphi(\SS)}$.
 For~$\tt>0$, the polynomial
\begin{multline*}
 D^{\tt}(s)\cdot b_{1}...b_{\tt}
 = (-1)^{|b_{1}||s|}b_{1}\cdot D^{\tt}(s)\cdot b_{2}\cdots   b_{\tt}\\
 = (-1)^{|b_{1}||s|}(b_{1}\circ(D^{\tt-1}(s)\cdot b_{2}...   b_{\tt}))
 -(-1)^{|b_{1}||s|}b_{1}\cdot \sum_{2\leq i \leq \tt}(D^{\tt-1}(s)\cdot b_{2}...\bb_{i-1}\cdot (Db_{i})\cdot b_{i+1}... b_{\tt})\\
 =  (-1)^{|b_{1}||s|}(b_{1}\circ(D^{\tt-1}(s)\cdot b_{2}...b_{\tt}))
 -\sum_{2\leq i \leq \tt} (-1)^{|b_{i}||b_{i+1}\cdots   b_{\tt}|}((D^{\tt-1}(s) \cdot b_{1}...  b_{i-1}b_{i+1}...b_{\tt})\circ b_{i})
\end{multline*}
lies in~$\IDD{\varphi(\SS)}$ by induction hypothesis.

To prove~\ITEM2,   we use induction on~$\rr$.
For~$\rr=1$, we obtain~$\ff\cdot \oD(\cc)=\ff\circ \cc\in\IDD{\varphi(\SS)}$.
For~$\rr>1$, the polynomial
\begin{multline*}
 f\cdot a_{1}...a_{r-1} D^{r}(c)
 = (\ff\circ(a_{1}...a_{r-1} D^{r-1}(c)))
 -\sum_{1\leq i\leq \rr-1}(\ff\cdot a_{1}...\aa_{i-1}\cdot \oD\aa_i \cdot \aa_{i+1}...a_{r-1} D^{r-1}(c))\\
 =(\ff\circ(a_{1}...a_{r-1} D^{r-1}(c)))
 -\sum_{1\leq i\leq \rr-1}(-1)^{|\aa_i||\aa_{i+1}...\aa_{r-1}c|}((\ff\cdot a_{1}...\aa_{i-1} \aa_{i+1}...a_{r-1} D^{r-1}(c))\circ \aa_i)
\end{multline*}
lies in~$\IDD{\varphi(\SS)}$ by induction hypothesis.
\end{proof}

We then have the following Poincar\'{e}-Birkhoff-Witt type theorem.
\begin{theorem}
Every GDN superalgebra $\ns\XX\SS$ can be  embedded into
its universal enveloping associative differential supercommutative algebra~$\kds{\varphi(\SS)}$, where~$$\varphi  :\fns\XX\longrightarrow \kk\DX$$ is the GDN superalgebra
 homomorphism induced by~$\varphi(\aa)=\aa$ for every~$\aa$ in~$\XX$.
\end{theorem}
\begin{proof}
By Lemmas~\ref{isomorphism} and~\ref{Novikov ideal}, we obtain
\begin{multline*}
  \ns\XX\SS\cong \kdds{\varphi(\SS)}=\frac{(\kk\DDX,\circ)}{\IDD{\varphi(\SS)} }=\frac{(\kk\DDX, \circ)}{\IdD{\varphi{(\SS)}}\cap \kk\DDX} \\
 \cong \frac{\kk\DDX+\IdD{\varphi{(\SS)}}}{\IdD{\varphi{(\SS)}} } \leq  \frac{(\kk\DX, \circ)}{\IdD{\varphi{(\SS)}}} =\kds{\varphi(\SS)}.
\end{multline*}
The lemma follows.
\end{proof}

\section{Engel Theorem}\label{nil-sec}
Our aim in this section is to prove an Engel theorem for GDN superalgebras (Theorem~\ref{Engel-theorem}), which is based on what was done for GDN algebras~\cite{engel novikov}. In this section, we assume that  the characteristic~$\chart(\kk)$  of the field~$\kk$ is~$0$, and assume that~$\mA=\mA_0\oplus\mA_1$ is a GDN superalgebra.

For every~$\xx$ in~$(\mA,\circ)$, let~$\rho_{\xx}$ be the right multiplication operator
$$\rho_\xx: \mA\longrightarrow \mA, \  \rho_\xx(\yy)= (\yy\circ \xx) \mbox{ for every }\yy \mbox{ in }\mA.$$ Then~$\mA$ is called right-nil of bound index if, for some positive integer~$n$, for every~$x\in \mA$, we have~$\rho_{x}^{n-1}(x)=0$.
We use the notation~$\xx^n\LSUB$ for~$\rho_{x}^{n-1}(x)$.
For all~$\xx_1\wdots\xx_n$ in~$\mA$, define
$$
\lnormed{\xx_1}{\xx_n}= ((...((x_1\circ x_2) \circ \xx_3)\circ\pdots \circ)\circ x_n) \mbox{ (left-normed bracketing).}
$$
 For all subspace~$\mV_1\wdots\mV_n$ of~$\mA$, define
$$
[\mV_1\wdots\mV_n]\LSUB= \mathsf{span}_{k}\{ \lnormed{\xx_1}{\xx_n} \mid  x_i \in \mV_{i}, 1\leq i\leq n \}.
$$
In particular, we obtain
$$
 \mV^n\LSUB=[\underbrace{\mV\wdots\mV}_{\mbox{$\nn$ times}}]\LSUB \mbox{ and } [\mV_1,\mV_2]\LSUB=(\mV_1\circ \mV_2)=\mathsf{span}_{\kk}\{(\xx_1\circ\xx_2)\mid \xx_1\in \mV_1, \xx_2\in\mV_2\}.
$$
We call an algebra~$\mA$  right-nilpotent if $\mA^{n}\LSUB=0$ for some positive integer~$n$.
Finally, for every subspace~$\mV$ of~$\mA$, for every integer~$n\geq 1$, define the subspace~$\mV^n$ of~$\mA$ inductively as follows:

  \ITEM1 $\mV^1=\mV$ and~$\mV^2=(\mV \circ \mV)$;

  \ITEM2 $\mV^n=\sum_{1\leq i\leq n-1}(\mV^i \circ \mV^{n-i})$.\\
We call an algebra~$\mA$ nilpotent if $\mA^{n}=0$ for some positive integer~$n$.

Since~$\mA_0$ is an ordinal GDN algebra, by Lemmas~6 and~7 in~\cite{engel novikov}, we get the following lemma, which shows that every right-nil GDN algebra of bound index is right nilpotent.  For the convenience of the readers, we quickly repeat the argument.
\begin{lemm}\cite{engel novikov}\label{engel theorem}
Let~$\mA=\mA_0\oplus\mA_1$ be a GDN superalgebra over a field of characteristic~$0$. If for some positive integer~$n$, for every~$\xx\in \mA_{0}$,
we have~$\xx^{n}\LSUB=0$, then $(\mA_0)^{n+1}\LSUB=0$.
\end{lemm}

\begin{proof}  For all~$\xx_1\wdots\xx_t$ in~$\mA_{0}$, define
$$
S(\xx_1,\xx_2,\dots, \xx_t)=\sum_{\sigma \in S_{t}} \lnormed{\xx_{\sigma (1)},\xx_{\sigma (2)}}{\xx_{\sigma (t)}},
$$
where~$S_t$ is the symmetric group of order~$t$.
Then for every term~$\mu$ occurred in the polynomial
$(\xx_1+\xx_2+\cdots + \xx_t)^{t}\LSUB-S(\xx_1,\xx_2,\dots, \xx_t)$,
there is some integer~$\ii\leq \tt$ such that the letter~$\xx_i$ does not occur in~$\mu$.
By the inclusion-exclusion properties, we get
$$
(\xx_1+\xx_2+\cdots + \xx_t)^{t}\LSUB-S(\xx_1,\xx_2,\dots, \xx_t)=
\sum_{\varnothing \neq \{i_1,i_2,\cdots i_r\}\subsetneqq \{1,\dots, t\}}(-1)^{t-r+1}(\xx_{i_1}+\cdots + \xx_{i_r})^{t}\LSUB.
$$
Therefore,  for~$\tt\geq n$, we get~$S(\xx_1,\xx_2,\dots, \xx_t)=0$.
Moreover, using right (super)commutativity, it is straightforward to show that
$$
S(\xx_1,\dots, \xx_{t+1})
=t(S(\xx_2,\dots, \xx_{t+1}) \circ \xx_{1})+
t!\lnormed{\xx_1}{\xx_{t+1}}.
$$
Since~$\chart(\kk)=0$, we have  $\lnormed{\xx_1}{\xx_{t+1}}=0$ for
every~$t\geq n+1$.
\end{proof}

The following lemma shows that, if a GDN superalgebra~$\mA$ is right nilpotent, then~$\mA^2$ is nilpotent. This result is directly reminiscent of that for GDN algebras~\cite{engel novikov}.

\begin{lemm}\label{A2 in A}
Let $\mA$ be a  GDN superalgebra. Then for every positive integer~$n$, the space~$\mA^{n}\LSUB$ is an ideal of~$\mA$,
and we have~$\mathcal(\mA^{2})^{n}\subseteq \mA^{n+1}\LSUB$.
\end{lemm}

\begin{proof}
We first use induction on~$n$ to show that~$\mA^{n}\LSUB$ forms an ideal of~$\mA$. For~$n=1$, it is clear.
For~$\nn\geq 2$,  we have~$$(\mA^{\nn}\LSUB\circ \mA)= \mA^{\nn+1}\LSUB \subseteq \mA^{\nn}\LSUB,$$
and by induction hypothesis, we also have
$$(\mA \circ \mA^{\nn}\LSUB)\subseteq ((\mA \circ \mA^{\nn-1}\LSUB)\circ \mA)
+(\mA^{\nn-1}\LSUB\circ (\mA\circ \mA))+((\mA^{\nn-1}\LSUB\circ \mA)\circ \mA) \subseteq \mA^{\nn}\LSUB.$$

Now we use induction on~$\nn$ to show~$ (\mA ^{2})^{n}\subseteq \mA^{n+1}\LSUB$. For~$n=1$, it is clear.
For~$n\geq 2$, we obtain
\begin{multline*}
(\mA^{2})^{\nn}=\sum_{1\leq i \leq \nn-1}((\mA^{2})^{i}\circ (\mA^{2})^{\nn-i})
\subseteq \sum_{1\leq i\leq \nn-1}(\mA^{i+1}\LSUB\circ \mA^{\nn-i+1}\LSUB)\\
\subseteq\sum_{1\leq i\leq \nn-1}[\mA,\mA^{\nn-i+1}\LSUB ,\underbrace{\mA\wdots\mA}_{ \mbox{$i$ times}}]\LSUB
\subseteq\sum_{1\leq i\leq \nn-1}[\mA^{\nn-i+1}\LSUB ,\underbrace{\mA\wdots\mA}_{ \mbox{$i$ times}}]\LSUB
= \mA^{\nn+1}\LSUB.
\end{multline*}
The proof is completed.
\end{proof}

We want to show that, under certain conditions, every right-nil GDN superalgebra of bounded index is right nilpotent. And the main difficulty lies in how to deal with the space~$[\mA_0, \mA_1, \mA_1\wdots \mA_1]\LSUB$.  The idea is to ``split"~$\mA_1$ in the following sense.

\begin{lemm}\label{make 01}
Let~$\mA=\mA_0\oplus \mA_1$ be a GDN superalgebra generated by~$\XX=\XX_0\cup \XX_1$. Then for every integer~$q\geq 1$, we have
$$
[\mA_0,\underbrace{\mA_1\wdots\mA_1}_{ \mbox{ $q$ times}}]\LSUB
 \subseteq
  \sum_{2t+m+p=q,t,m\geq 0, p\in \{1,2\}}
 [\mA_0, \underbrace{\mA_0\wdots \mA_0}_{ \mbox{ $t$ times }}, \underbrace{kX_1\wdots kX_1}_{ \mbox{ $m$ times }},\underbrace{\mA_1\wdots \mA_1}_{\mbox{ $p$ times }}]\LSUB,
$$
where~$\kk\XX_1$ is the subspace of~$\mA_1$ spanned by~$\XX_1$.
\end{lemm}

\begin{proof}
We use induction on~$q$. For~$q\leq 2$, it is clear.
For~$q=3$, we shall show that
$$
[\mA_0, \mA_1,\mA_1,\mA_1]\LSUB \subseteq [\mA_0,\mA_1,\mA_0]\LSUB
+[\mA_0, \mA_1,\mA_1,\kk\XX_1]\LSUB.
$$
It is enough to show that, for every term~$\mu$ over~$\XX$ of parity~0, for all terms~$\mu_1$, $\mu_2$ and~$\mu_3$ over~$\XX$ of parity~1, we have
$$
[\mu, \mu_1 , \mu_2, \mu_3]\LSUB \in [\mA_0,\mA_1,\mA_0]\LSUB+[\mA_0,\mA_1,\mA_1,\kk\XX_1]\LSUB.
$$
We use induction on~$\ell(\mu_1)$. For~$\ell(\mu_1)=1$, the claim follows by right supercommutativity.
For~$\ell(\mu_1)>1$, suppose that~$\mu_1=(\mu_{11}\circ \mu_{12})$.   If~$\mu_{12}$ lies in~$\mA_1$, then~$\mu_{11}$ lies in~$\mA_0$, and by induction hypothesis, we have
$[\mu,\mu_{12}, (\mu_{11}\circ \mu_2),\mu_3]\LSUB
\in [\mA_0,\mA_1,\mA_0]\LSUB+[\mA_0,\mA_1,\mA_1,kX_1]\LSUB$.
Therefore, we obtain
\begin{multline*}
[\mu, \mu_1 , \mu_2, \mu_3]\LSUB
=[\mu, (\mu_{11}\circ \mu_{12}), \mu_2, \mu_3]\LSUB\\
=[\mu, \mu_{11}, \mu_{12},\mu_2, \mu_3]\LSUB-[\mu_{11},\mu,\mu_{12},\mu_2,\mu_3]\LSUB
+[\mu_{11}, (\mu \circ \mu_{12}), \mu_2,\mu_3]\LSUB\\
 =[\mu, \mu_{12},\mu_2,\mu_3,\mu_{11}]\LSUB-[\mu_{11},\mu_{12}, \mu_2,\mu_3, \mu]\LSUB
+[\mu_{11},(\mu \circ \mu_{12}),\mu_2,\mu_3]\LSUB\\
 =[\mu, \mu_{12},\mu_2,\mu_3,\mu_{11}]\LSUB-[\mu_{11},\mu_{12}, \mu_2,\mu_3, \mu]\LSUB
+[\mu_{11},((\mu \circ \mu_{12})\circ\mu_2), \mu_3]\LSUB\\
+[(\mu\circ\mu_{12}),\mu_{11},\mu_2,\mu_3]\LSUB
-[(\mu\circ\mu_{12}),(\mu_{11}\circ \mu_2),\mu_3]\LSUB\\
 =[((\mu\circ \mu_{12})\circ\mu_2),\mu_3,\mu_{11}]\LSUB-[((\mu_{11}\circ\mu_{12})\circ \mu_2),\mu_3, \mu]\LSUB\\
+[\mu_{11}, \mu_3,((\mu \circ \mu_{12})\circ\mu_2)]\LSUB
+[((\mu\circ \mu_{12})\circ\mu_2),\mu_3,\mu_{11}]\LSUB\\
-[\mu,\mu_{12},(\mu_{11}\circ \mu_2),\mu_3]\LSUB
\in [\mA_0, \mA_1,\mA_0]\LSUB+[\mA_0,\mA_1,\mA_1,kX_1]\LSUB.
\end{multline*}
 If~$\mu_{12}$ lies in~$\mA_0$, then~$\mu_{11}$ lies in~$\mA_1$, and we obtain
\begin{multline*}
[\mu, \mu_1 , \mu_2, \mu_3]\LSUB
=[\mu, (\mu_{11}\circ \mu_{12}), \mu_2, \mu_3]\LSUB\\
=[\mu, \mu_{11}, \mu_{12},\mu_2, \mu_3]\LSUB-[\mu_{11},\mu,\mu_{12},\mu_2,\mu_3]\LSUB
+[\mu_{11}, (\mu \circ \mu_{12}), \mu_2,\mu_3]\LSUB\\
=[((\mu\circ \mu_{11})\circ \mu_2),  \mu_3 ,\mu_{12}]\LSUB+[(\mu_{11} \circ \mu_2),\mu_3,(\mu \circ \mu_{12})]\LSUB
-[((\mu_{11}\circ \mu)\circ \mu_2), \mu_3,\mu_{12}]\LSUB.
\end{multline*}
So~$[\mu, \mu_1 , \mu_2, \mu_3]\LSUB$ lies in~$[\mA_0,\mA_1,\mA_0]\LSUB$.

For~$q\geq 4$, by right supercommutativity and the case~$\qq=3$, we have
\begin{multline*}
[\mA_0, \underbrace{\mA_1 \wdots \mA_1}_{\mbox{$q$ times}}]\LSUB
\subseteq
[\mA_0, \underbrace{\mA_1 \wdots \mA_1,}_{\mbox{$q-2$ times}}\mA_0]\LSUB
+
[\mA_0, \underbrace{\mA_1 \wdots \mA_1,}_{\mbox{$q-1$ times}} kX_1]\LSUB\\
\subseteq
\sum_{2t+m+p=q-2,t,m\geq 0, p\in \{1,2\}}[\mA_0, \underbrace{\mA_0 \wdots \mA_0, }_{\mbox{$t$ times}} \underbrace{kX_1 \wdots kX_1}_{\mbox{$m$ times}}
,\underbrace{\mA_1 , \mA_1,}_{\mbox{$p$ times}}\mA_0]\LSUB\\
+ \sum_{2t+m+p=q-1,t,m\geq 0, p\in \{1,2\}}
[\mA_0, \underbrace{\mA_0 \wdots \mA_0, }_{\mbox{$t$ times}} \underbrace{kX_1 \wdots kX_1}_{\mbox{$m$ times}}
,\underbrace{\mA_1  , \mA_1}_{\mbox{$p$ times}}, kX_1]\LSUB\\
\subseteq \sum_{2t+m+p=q,t,m\geq 0, p\in \{1,2\}}
[\mA_0, \underbrace{\mA_0\wdots \mA_0 ,}_{\mbox{$t$ times}} \underbrace{kX_1 \wdots kX_1,}_{\mbox{$m$ times}}\underbrace{\mA_1  , \mA_1}_{\mbox{$p$ times}}]\LSUB.
\end{multline*}
The claim follows.
\end{proof}

We are now in a position to prove the following Engel theorem.

\begin{theorem}\label{Engel-theorem}
Let~$\mA=\mA_0\oplus \mA_1$ be a GDN superalgebra generated by~$\XX=\XX_0\cup \XX_1$ over a field of characteristic~0,
 where every element of the set~$\XX_0$ is of parity~0 and every element of the set~$\XX_1$ is of parity~1. If~$\XX_1$ is a finite set and the even part~$\mA_0$ is right-nil of bounded index~$\nn$, then~$\mA$ is right nilpotent,
 in particular, the ideal~$\mA^2$ of~$\mA$ is nilpotent.
\end{theorem}

\begin{proof}
By Lemma \ref{engel theorem}, we have  $(\mA_0)^{n+1}\LSUB=0$. Let $n_0= \max(\sharp(X_1), n+1)$ and let~${q=3n_0+1}$.
We shall show that~$\mA^{q}\LSUB=0$. It is enough to show the following two claims:
$$[\mA_0, \underbrace{\mA_1 \wdots
\mA_1}_{\mbox{$j$ times}}
, \underbrace{\mA_0 \wdots \mA_0 }_{\mbox{$q-2-j$ times}}]\LSUB=0\  \mbox{ for every integer } j \mbox{ such that } 0\leq j\leq q-2,$$
 and
$$[\mA_1, \underbrace{\mA_1 \wdots \mA_1}_{\mbox{$j$ times}}
, \underbrace{\mA_0 \wdots \mA_0 }_{\mbox{$q-1-j$ times}}]\LSUB=0 \ \mbox{ for every integer } j \mbox{ such that }  0\leq j\leq q-1.$$

For the first claim, if $j=0$, then by Lemma \ref{engel theorem}, we get~$(\mA_0)^{q-1}\LSUB=0 $.
For every integer~$\jj$ such that~$1\leq j\leq q-2$, by Lemma~\ref{make 01}, we have
\begin{eqnarray*}
&&[\mA_0, \underbrace{\mA_1 \wdots
\mA_1}_{\mbox{ $j$ times}}
, \underbrace{\mA_0 \wdots \mA_0 }_{ \mbox{$q-2-j$ times }}]\LSUB  \\
&\subseteq &\sum_{2t+m+p=j,t,m\geq 0, p\in \{1,2\}}
[\mA_0, \underbrace{\mA_0 \wdots \mA_0, }_{ \mbox{$t+q-2-j$  times }}
 \underbrace{kX_1 \wdots kX_1,}_{ \mbox{ $m$ times }}
\underbrace{\mA_1  , \mA_1}_{ \mbox{ $p$ times }}]\LSUB.
\end{eqnarray*}
If~$m>n_0, $ then we obtain
$$ [\mA_0, \underbrace{\mA_0 \wdots \mA_0, }_{ \mbox{$t+q-2-j$  times }}
\underbrace{kX_1 \wdots kX_1,}_{ \mbox{ $m$ times }}
\underbrace{\mA_1  , \mA_1}_{ \mbox{ $p$ times }}]\LSUB=0.$$
If~$m\leq n_0$, then we obtain~$t=\frac{1}{2}(j-m-p)$ and
\begin{multline*}
t+q-2-j
=\frac{1}{2}(j-m-p)+q-j-2\\
=\frac{1}{2}(-m-p+q)+\frac{1}{2}(q-j)-2
\geq \frac{1}{2}(-n_0-2+3n_0+1)+\frac{1}{2}\times 2 -2
=n_0-\frac{3}{2}.
\end{multline*}
Since~$t+q-2-j$ is an integer, we have~$t+q-2-j\geq n_0-1$.
So we get
$$ [\mA_0, \underbrace{\mA_0 \wdots \mA_0, }_{ \mbox{$t+q-2-j$  times }}
\underbrace{kX_1 \wdots kX_1,}_{ \mbox{ $m$ times }}
\underbrace{\mA_1  , \mA_1}_{ \mbox{ $p$ times }}]\LSUB=0.$$
The first claim follows.

 For the second claim, if~$j\geq 1$,  then by the first claim, we get
$$[\mA_1, \underbrace{\mA_1 \wdots \mA_1}_{\mbox{ $j$ times}}
, \underbrace{\mA_0 \wdots \mA_0 }_{\mbox{$q-1-j$ times}}]\LSUB
\subseteqq [\mA_0, \underbrace{\mA_1 \wdots \mA_1,}_{\mbox{$j-1$ times}}\
\underbrace{\mA_0  \wdots \mA_0 }_{ \mbox{$q-2-(j-1)$ times}}]\LSUB=0.$$
 If~$j=0$, then we first use induction on~$\qq$ to show that, for every~$\qq\geq 3$, we have
$$
[\mA_1, \underbrace{\mA_0 , \cdots , \mA_0 }_{ \mbox{$q-1$ times }}]\LSUB
\subseteq (\mA_1\circ {(\mA_0)}^{q-1}\LSUB)+({(\mA_0)}^{q-1}\LSUB\circ \mA_1)+({(\mA_0)}^{q-2}\LSUB\circ (\mA_1\circ\mA_0)).
$$
For~$\qq=3$, by left supersymmetry and right supercommutativity, we get
\begin{multline*}
[\mA_1, \mA_0, \mA_0 ]\LSUB
\subseteq (\mA_1\circ (\mA_0 \circ \mA_0))+((\mA_0 \circ \mA_1)\circ \mA_0)+(\mA_0 \circ (\mA_1\circ\mA_0))\\
\subseteq (\mA_1\circ (\mA_0 \circ \mA_0))+((\mA_0 \circ \mA_0)\circ \mA_1)+(\mA_0 \circ (\mA_1\circ\mA_0)).
\end{multline*}
 For~$\qq>3$, by induction hypothesis, we get
 \begin{multline*}
  [\mA_1, \underbrace{\mA_0 , \cdots , \mA_0 }_{ \mbox{$q-1$ times }}]\LSUB
  =([\mA_1, \underbrace{\mA_0 , \cdots , \mA_0 }_{ \mbox{$q-2$ times }}]\LSUB\circ\mA_0)\\
\subseteq ((\mA_1\circ {(\mA_0)}^{q-2}\LSUB)\circ \mA_0)+(({(\mA_0)}^{q-2}\LSUB\circ \mA_1)\circ\mA_0)+(({(\mA_0)}^{q-3}\LSUB\circ (\mA_1\circ\mA_0))\circ\mA_0) \\
\subseteq (\mA_1\circ ({(\mA_0)}^{q-2}\LSUB\circ \mA_0))+(({(\mA_0)}^{q-2}\LSUB\circ \mA_1)\circ\mA_0)+({(\mA_0)}^{q-2}\LSUB\circ (\mA_1\circ\mA_0))\\
\subseteq (\mA_1\circ {(\mA_0)}^{q-1}\LSUB)+({(\mA_0)}^{q-1}\LSUB\circ \mA_1)+({(\mA_0)}^{q-2}\LSUB\circ (\mA_1\circ\mA_0)).
\end{multline*}
  Finally, since~${(\mA_0)}^{q-2}\LSUB=0$, the second claim follows.
\end{proof}


\begin{thebibliography}{15}
\addcontentsline{toc}{chapter}{Bibliography}
\bibitem{Bai 1}C. Bai and D. Meng,   The classification of Novikov algebras in low dimensions,  J. Phys. A: Math. Gen.  34(8) (2001)   1581--1594. 

\bibitem{Bai 3}C. Bai and D. Meng,    Transitive Novikov algebras on four-dimensional nilpotent Lie algebras,   International Journal of Theoretical Physics    40(10) (2001)   1761--1768. 

\bibitem{Novikov1}A.A. Balinskii and S.P. Novikov, Poisson brackets of hydrodynamics type, Frobenius algebras and Lie algebras (Russian),
Dokl. Akad. Nauk SSSR 283(5) (1985) 1036--1039. 

\bibitem{GDN}L.A. Bokut, Y.Q. Chen and Z. Zhang, Gr\"{o}bner--Shirshov bases method
for Gelfand--Dorfman--Novikov algebras, Journal of Algebra and Its Applications,
 16(1) (2017) 1750001 (22 pages), DOI: 10.1142/S0219498817500013 

\bibitem{GDNP}L.A. Bokut, Y.Q. Chen and Z. Zhang, On free Gelfand--Dorfman--Novikov--Poisson algebras and a PBW theorem, Journal of Algebra 500 (2018) 153--170. 

\bibitem{Burde 3}D. Burde and K. Dekimpe,  Novikov structures on solvable Lie algebras, Journal of Geometry and Physics 56(9) (2006)  1837-1855. 


\bibitem{Chen ly}L. Chen,  Y. Niu and D. Meng,  Two kinds of Novikov algebras and their realizations,  Journal of Pure and Applied Algebra  212(4) (2008)  902--909. 

\bibitem{trees}A.S.  Dzhumadildaev and C. L\"{o}fwall,   Trees,  free right-symmetric algebras,  free Novikov algebras and identities,  Homology, Homotopy  and Applications  4(2) (2002)  165--190.

\bibitem{engel novikov}A.S. Dzhumadil’daev and K.M. Tulenbaev, Engel theorem for Novikov algebras,
 Communications in Algebra 34(3) (2006)  883--888.

\bibitem{Filippov 2}V.T. Filippov, On right-symmetric and Novikov nil algebras of bounded index, (Russian. Russian summary) Mat. Zametki 70(2) (2001)  289--295; translation in
Math. Notes 70(2) (2001)  258-263. 

\bibitem{Gelfand}I.M. Gelfand and I.Ya. Dorfman,  Hamiltonian operators and algebraic structures related to them,  Functional Analysis and Its Applications  13(4) (1979)  248--262. 

\bibitem{super low dimension}Y. Kang and Z. Chen, Novikov superalgebras in low dimensions,
Journal of Nonlinear Mathematical Physics 16(3) (2009)  251--257.

\bibitem{Novikov2}S.P. Novikov, The geometry of conservative systems of hydrodynamic type. The method of averaging for field-theoretical systems , Uspekhi Mat. Nauk  40(4)(244) (1985)  79--89; Russian Math. Surveys, 40(4) (1985) 85--98. 

\bibitem{Osborn 5}J.M. Osborn and  E.I. Zelmanov,  Nonassociative algebras related to Hamiltonian operators in the formal calculus of variations, Journal of Pure and Applied Algebra  101(3) (1995)  335--352. 

\bibitem{Xiaoping Xu 2}X. Xu,  On simple Novikov algebras and their irreducible modules,  Journal of Algebra  185(3) (1996) 905--934.

\bibitem{Xiaoping Xu conformal}X. Xu,   Quadratic Conformal Superalgebras, Journal of Algebra  231(1) (2000)  1--38. 

\bibitem{Xiaoping Xu super}X. Xu, Variational calculus of supervariables and related algebraic structures,
Journal of Algebra 223(2) (2000)  396--437.

\bibitem{Zelmanov}E.I. Zelmanov,  On a class of local translation invariant Lie algebras,   Soviet Math. Dokl. 35(1)  (1987)  216--218.

\bibitem{A001}F. Zhu and Z. Chen, Novikov superalgebras with $\mA_0=\mA_1\mA_1$,
 Czechoslovak Mathematical Journal  60(4) (2010)   903--907. 
\end{thebibliography}
\end{document}